\documentclass[a4paper]{article}
\usepackage{url}
\usepackage{amsmath,amsfonts,amssymb,amsthm,graphicx,enumerate,mathrsfs,color,subfigure,multicol,mathdots,booktabs,xypic}
\usepackage[noend]{algpseudocode}
\usepackage{algorithmicx,algorithm}
\newtheorem{definition}{Definition}
\newtheorem{remark}{Remark}
\newtheorem{theorem}{Theorem}
\newtheorem{corollary}{Corollary}
\newtheorem{lemma}{Lemma}
\newtheorem{example}{Example}
\newcommand{\ten}[1]{\mathcal{#1}}
\newcommand{\tens}[1]{\mathcal{#1}}

\newcommand{\bcirc}{{\rm{bcirc}}}
\newcommand{\unfold}{{\rm{unfold}}}
\newcommand{\fold}{{\rm{fold}}}

\newcommand{\norm}[1]{\left\lVert#1\right\rVert}
\begin{document} \large
\title{Generalized Tensor Function via the Tensor Singular Value Decomposition based on the T-Product}
\author{Yun Miao\footnote{E-mail: 15110180014@fudan.edu.cn. School of Mathematical Sciences, Fudan University, Shanghai, 200433, P. R. of China. Y. Miao is supported by the National Natural Science Foundation of China under grant 11771099. } \quad  Liqun Qi\footnote{ E-mail: maqilq@polyu.edu.hk. Department of Applied Mathematics, the Hong Kong Polytechnic University, Hong Kong. L. Qi is supported by the Hong Kong Research Grant Council (Grant No. PolyU 15302114, 15300715, 15301716 and 15300717)} \quad  Yimin Wei\footnote{Corresponding author. E-mail: ymwei@fudan.edu.cn and yimin.wei@gmail.com. School of Mathematical Sciences and Shanghai Key Laboratory of Contemporary
Applied Mathematics, Fudan University, Shanghai, 200433, P. R. of China. Y. Wei is supported by the Innovation Program of Shanghai Municipal Education Commission.
}}
\maketitle

\begin{abstract}
In this paper, we present the definition of generalized tensor function according to the tensor singular value decomposition (T-SVD) based on the tensor T-product. Also, we introduce the compact singular value decomposition (T-CSVD) of tensors, from which the projection operators and Moore Penrose inverse of tensors are obtained. We establish the Cauchy integral formula for tensors by using the partial isometry tensors and applied it into the solution of tensor equations. Then we establish the generalized tensor power and the Taylor expansion of tensors. Explicit generalized tensor functions are listed. We define the tensor bilinear and sesquilinear forms and proposed theorems on structures preserved by generalized tensor functions. For complex tensors, we established an isomorphism between complex tensors and real tensors. In the last part of our paper, we find that the block circulant operator established an isomorphism between tensors and matrices. This isomorphism is used to prove the F-stochastic structure is invariant under generalized tensor functions. The concept of invariant tensor cones is raised.

\bigskip

  \hspace{-14pt}{\bf Keywords.} T-product, T-SVD, T-CSVD, generalized tensor function, Moore-Penrose inverse, Cauchy integral formula, tensor bilinear form, Jordan algebra, Lie algebra, block tensor multiplication, complex-to-real isomorphism, tensor-to-matrix isomorphism.

  \bigskip

  \hspace{-14pt}{\bf AMS Subject Classifications.} 15A48, 15A69, 65F10, 65H10, 65N22.

\end{abstract}

\newpage
\section{Introduction}
Matrix functions have wide applications in many fields. They emerge as exponential integrators in differential equations. For square matrices, people usually define the matrix function by using its Jordan canonical form \cite{Golub1,Higham1}. Unfortunately, this kind of method could not be extended to rectangular matrices. In 1972, Hawkins and Ben-Israel \cite{Hawkins1} (or \cite[Chapter 6]{Ben1}) first introduced the generalized matrix functions by using the singular value decomposition (SVD) and compact singular value decomposition (CSVD) for rectangular matrices. It has been recognized that the generalized matrix functions have great uses in data science, matrix optimization problems, Hamiltonian dynamical systems and etc. Recently, Benzi et al. considered the structural properties which are preserved by generalized matrix functions \cite{Arrigo1,Aurentz1,Benzi1}. Noferini \cite{Noferini1}  provided the formula for the Fr${\rm \acute{e}}$chet derivative of a generalized matrix function.

As high-dimension analogues of matrices, a tensor means a hyperdimensional matrix, they are extensions of matrices. The difference is that a matrix entry $a_{ij}$ has two subscripts $i$ and $j$, while a tensor entry $a_{i_1\dots i_m}$ has $m$ subscripts $i_1,\ldots,i_m$. We call $m$ the order of tensor, if the tensor has $m$ subscripts. Let $\mathbb{C}$ be the complex field and $\mathbb{R}$ be the real field. For a positive integer $N$, let $[N]=\{1,2,\ldots, N\}$. We say a tensor is a real tensor if all its entries are in $\mathbb{R}$ and a tensor is a complex tensor if all its entries are in $\mathbb{C}$.\par

Recently, the tensor T-product has been established and proved to be a useful tool in many areas, such as image processing \cite{Khaleel1, Kilmer1, Martin1, Soltani1, Tarzanagh1, Zhou1}, computer vision \cite{Baburaj1, Hao1, Xie1,Yin1}, signal processing \cite{Chan2, Liu1, Long1, Sun1}, low rank tensor recovery and robust tensor PCA \cite{Kong2,Liu1},  data completion and denoising \cite{Ely1, Hu1, Hu2, Liu2,Long1, Madathil1, Madathil2, Qin1, Wang1, Yang1, CZhang1, Zhang1, Zhang2, Zhang3}.
Because of the importance of tensor T-product, Lund \cite{Lund1} gave the definition for tensor functions based on the T-product of third-order F-square tensors which means all the front slices of a tensor is square matrices. The definition of T-function is given by
$$
f(\tens{A}):=\fold(f(\bcirc(\tens{A}))\widehat{E_{1}}^{np\times n}),
$$
where `$\bcirc(\tens{A})$' is the block circulant matrix \cite{Chan1} by the F-square tensor $\tens{A}\in \mathbb{R}^{n\times n \times p}$ and see the detail in Section 2.2.\par
For the Einstein product, both the standard tensor inverse and the generalized tensor inverse theories have been established \cite{Ma1, 2Sun1}, so it is natural to talk about the generalized inverse and other kinds of generalized functions based on the T-product.\par
In this paper, we generalize the tensor T-function from F-square third order tensors to rectangular tensors $\tens{A}\in \mathbb{R}^{m \times n \times p}$ by using tensor singular value decomposition (T-SVD) and tensor compact singular value decomposition (T-CSVD). Kilmer \cite{Kilmer2} gives the tensor singular value decomposition in 2011 (See Fig. \ref{fig1-1}), which gives a new tensor representation and compression idea based on the tensor T-product method especially for third order tensors. The tensor singular value decomposition of tensor $\tens{A}\in \mathbb{C}^{m \times n \times p}$ is given by
\cite{Hao1,Kilmer1, Kilmer2}
$$
\tens{A}=\tens{U}*\tens{S}*\tens{V}^{H},
$$
where $\tens{U}\in \mathbb{C}^{m \times m \times p}$ and $\tens{V}\in \mathbb{C}^{n \times n \times p}$ are unitary tensors and $\tens{S}\in \mathbb{C}^{m \times n \times p}$ is a F-diagonal tensor respectively. The entries in $\tens{S}$ are called the singular tubes of $\tens{A}$.
 \begin{figure}\label{fig1-1}
  \centering
  \includegraphics[width=1\textwidth]{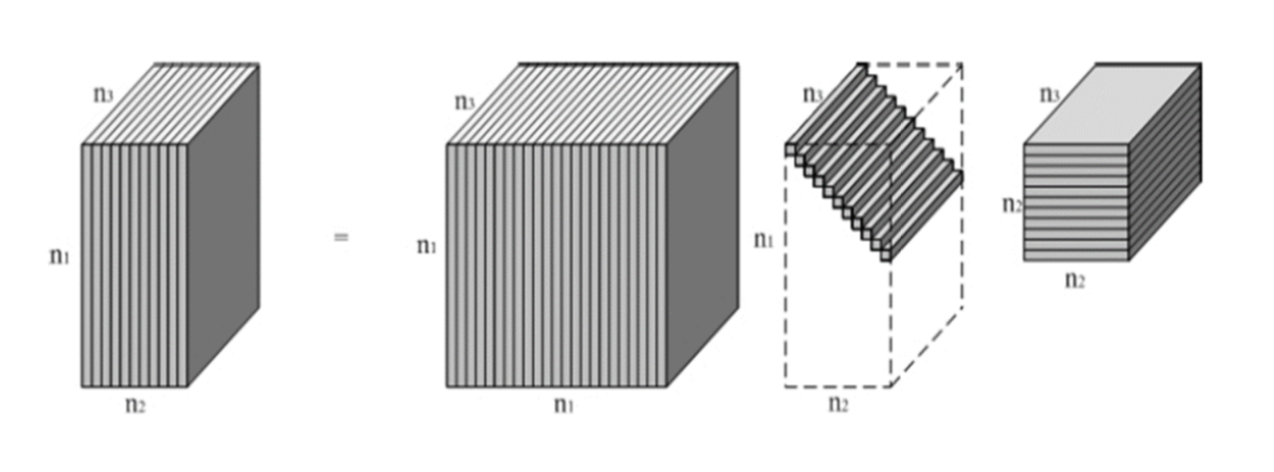}
  \caption{T-SVD of Tensors}
\end{figure}
By using this kind of decomposition, the definition of general tensor functions can be raised.\par
This paper is organized as follows. We make some review of the definition of tensor T-product and some algebraic structure of third order tensors via this kind of product in Preliminaries. Then we recall the definition of T-function given by Lund \cite{Lund1} and some of its properties. In the main part of our paper, we introduce the definition of tensor singular value decomposition and the generalized matrix functions. Then we extend the generalized matrix functions to tensors. Properties are given in the following part. In order to illustrate the generalized tensor functions explicitly, we present the definition of tensor compact singular value decomposition and tensor rank. Orthogonal projection tensors and Cauchy integral formula of tensor functions are provided. As a special case of generalized tensor functions, the expression of the Moore-Penrose inverse and the resolvent of a tensor are also introduced, which have the applications to give the solution of the tensor equation
$$
\tens{A}*\tens{X}*\tens{B}=\tens{D}.
$$
We give the definition of tensor power by using orthogonal projection. Taylor expansion of some explicit tensor function are listed. It should be noticed that, for simplicity of illustration, we only propose results for third order tensors. Results for $n$-th order tensors can also be deduced by our methods, see \cite{Liu3}. By establishing the block tensor multiplication based on the T-product, we summarize many kinds of special tensor structures which are preserved by generalized tensor functions, which is mainly in multiplicative group $\mathbb{G}$, Lie algebra $\mathbb{L}$ and Jordan algebra $\mathbb{J}$. Centrohermitian structure and block circulant structure are also considered. Since isomorphism relations are very important relation in algebra, we further establish the complex-to-real isomorphism which commutes with the generalized tensor function, so complex tensor functions can be isomorphically changed to real tensor functions. In the last part of our paper, we find that the block circulant operator `$\bcirc$' establishes an isomorphism between matrices and tensors, which means the generalized functions of tensors commutes with the operator `$\bcirc$'. Now we can say the generalized matrix function becomes a special case of the generalized tensor function. As an application of this theorem, we defined the F-stochastic tensor structure and proved that this kind of tensor structure is invariant under the generalized tensor functions.

\section{Preliminaries}
\subsection{Notation and index}
   A new concept is proposed for multiplying third-order tensors, based on viewing a tensor as a stack of frontal slices. Suppose two tensors $\tens{A}\in \mathbb{R}^{m\times n \times p}$ and $\tens{B} \in \mathbb{R}^{n\times s \times p}$ and denote their frontal faces respectively as $A^{(k)}\in \mathbb{R}^{m\times n}$ and $B^{(k)}\in \mathbb{R}^{n\times s}$, $k=1,2,\ldots, p$. We also define the operations $\bcirc$, $\unfold$ and $\fold$ as \cite{Hao1,Kilmer1, Kilmer2},
$$
\bcirc(\tens{A}):=
\begin{bmatrix}
 A^{(1)} &  A^{(p)}  &  A^{(p-1)} & \cdots &  A^{(2)}\\

 A^{(2)} &  A^{(1)}  &  A^{(p)} & \cdots &  A^{(3)}\\

\vdots  & \ddots& \ddots & \ddots & \vdots\\

 A^{(p)} &  A^{(p-1)}  &  \ddots & A^{(2)} &  A^{(1)}\\
\end{bmatrix},\
\unfold (\tens{A}):=
\begin{bmatrix}
A^{(1)}\\

A^{(2)}\\

\vdots\\

 A^{(p)}\\
\end{bmatrix},
$$
and $\fold (\unfold(\tens{A})):=\tens{A}$. We can also define the corresponding inverse operation $\bcirc^{-1}: \mathbb{R}^{mp\times np}\rightarrow \mathbb{R}^{m\times n \times p}$ such that $\bcirc^{-1}(\bcirc({\tens{A}}))=\tens{A}$.\par
 \begin{figure}\label{fig1-2}
  \centering
  \includegraphics[width=1\textwidth]{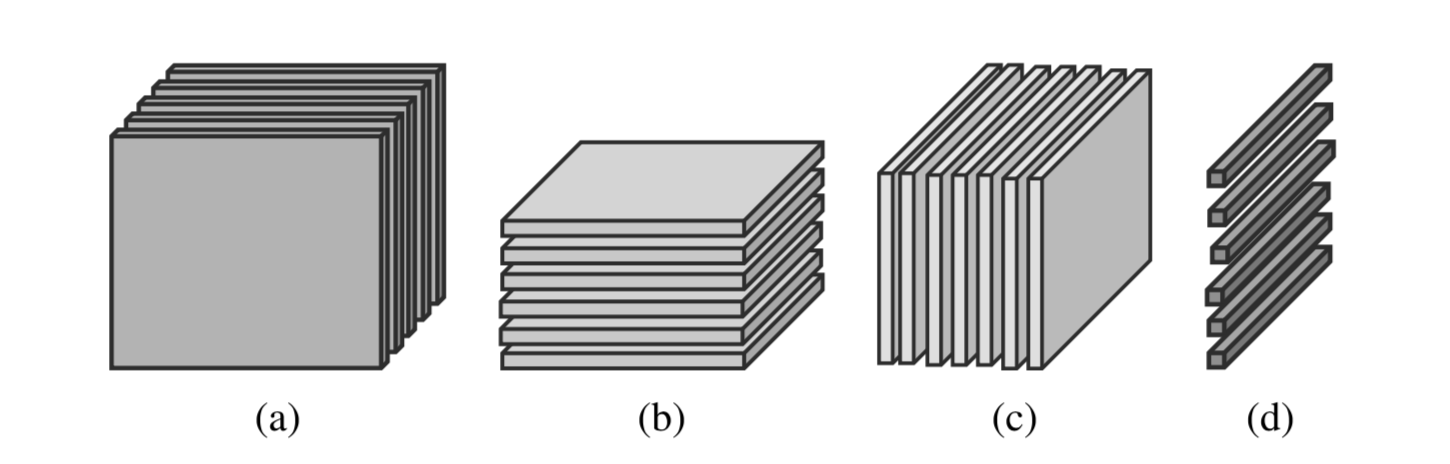}
  \caption{(a) Frontal, (b) horizontal, and (c) lateral slices of a 3rd order tensor. The lateral slices are also referred to as oriented matrices. (d) A lateral slice as a vector of tube fibers.}
\end{figure}

\subsection{The tensor T-Product}
The following definitions and properties are introduced in \cite{Hao1,Kilmer1, Kilmer2}.
\begin{definition}\label{def1-1} {\rm (T-product)}
Let $\tens{A}\in \mathbb{C}^{m\times n \times p}$ and $\tens{B}\in \mathbb{C}^{n\times s \times p}$ be two real tensors. Then the T-product
$\tens{A}*\tens{B}$ is a $m\times s \times p$ real tensor defined by
$$
\tens{A}*\tens{B}:=\fold (\bcirc(\tens{A})\unfold(\tens{B})).
$$
\end{definition}
We introduce definitions of transpose, identity and orthogonal of tensors as follows.
\begin{definition}\label{def1-2} {\rm(Transpose and conjugate transpose)}
If $\tens{A}$ is a third order tensor of size $m\times n\times p $, then the transpose $\tens{A}^{\top}$ is obtained by transposing each of the frontal slices and then reversing the order of transposed frontal slices $2$ through $p$. The conjugate transpose $\tens{A}^{H}$ is obtained by conjugate transposing each of the frontal slices and then reversing the order of transposed frontal slices $2$ through $p$.
\end{definition}
\begin{example}Suppose $\tens{A}\in \mathbb{R}^{3\times 3\times 3}$ is a real tensor whose elements of the frontal slices are given as:
$$
A^{(1)}=
\begin{bmatrix}
1&2&3\\
4&5&6\\
7&8&9\\
\end{bmatrix},\quad
A^{(2)}=
\begin{bmatrix}
2&4&6\\
8&10&12\\
14&16&18\\
\end{bmatrix},\quad
A^{(3)}=\begin{bmatrix}
3&6&9\\
12&15&18\\
21&24&27\\
\end{bmatrix}.
$$
Then by the definition of tensor transpose, the elements of the frontal slices of $\tens{A}^{\top}$ is given as:
$$
A^{(1)\top}=
\begin{bmatrix}
1&4&7\\
2&5&8\\
3&6&9\\
\end{bmatrix},\quad
A^{(2)\top}=
\begin{bmatrix}
3&12&21\\
6&15&24\\
9&18&27\\
\end{bmatrix},\quad
A^{(3)\top}=\begin{bmatrix}
2&8&14\\
4&10&16\\
6&12&18\\
\end{bmatrix}.
$$
\end{example}
\begin{definition}\label{def1-3} {\rm (Identity tensor)}
The $n\times n \times p $ identity tensor $\tens{I}_{nnp}$ is the tensor whose first frontal slice is the $n\times n$ identity matrix, and whose other frontal slices are all zeros.
\end{definition}
It is easy to check that $\tens{A}*\tens{I}_{nnp}=\tens{I}_{mmp}*\tens{A}=\tens{A}$ for $\tens{A}\in \mathbb{R}^{m\times n\times p}$.
\begin{definition}\label{def1-4} {\rm (Orthogonal and unitary tensor)}
An $n\times n\times p$ real-valued tensor $\tens{P}$ is orthogonal if $\tens{P}^{\top}*\tens{P}=\tens{P}*\tens{P}^{\top}=\tens{I}$.  An $n\times n\times p$ complex-valued tensor $\tens{Q}$ is unitary if $\tens{Q}^{H}*\tens{Q}=\tens{Q}*\tens{Q}^{H}=\tens{I}$.
\end{definition}
For a frontal square tensor $\tens{A}$ of size $n\times n \times p$, it has inverse tensor $\tens{B}(=\tens{A}^{-1})$, provided that
$$
\tens{A}*\tens{B}=\tens{I}_{nnp}\ \ and \ \ \tens{B}*\tens{A}=\tens{I}_{nnp}.
$$
It should be noticed that invertible third order tensors of size $n\times n\times p$ forms a group, since the invertibility of tensor $\ten{A}$ is equivalent to the invertibility of the matrix $\bcirc(\tens{A})$, and the set of invertible matrices forms a group. Also, the orthogonal tensors based on the tensor T-product also forms a group, since $\bcirc(\tens{Q})$ is an orthogonal matrix.
\begin{example}Suppose $\tens{A}\in \mathbb{R}^{2\times 2\times 3}$ is a real tensor whose elements of the frontal slices are given as:
$$
A^{(1)}=
\begin{bmatrix}
1&-\frac{1}{3}\\
\frac{1}{3}&1\\
\end{bmatrix},\quad
A^{(2)}=
\begin{bmatrix}
0&-\frac{1}{3}\\
\frac{1}{3}&0\\
\end{bmatrix},\quad
A^{(3)}=
\begin{bmatrix}
0&-\frac{1}{3}\\
\frac{1}{3}&0\\
\end{bmatrix}.
$$
It is easy to verify, under the tensor T-product, the elements of the frontal slices of the inverse tensor $\tens{B}=\tens{A}^{-1}$ is given as:
$$
B^{(1)}=
\begin{bmatrix}
\frac{5}{6}&\frac{1}{6}\\
-\frac{1}{6}&\frac{5}{6}\\
\end{bmatrix},\quad
B^{(2)}=
\begin{bmatrix}
-\frac{1}{6}&\frac{1}{6}\\
-\frac{1}{6}&-\frac{1}{6}\\
\end{bmatrix},\quad
B^{(3)}=
\begin{bmatrix}
-\frac{1}{6}&\frac{1}{6}\\
-\frac{1}{6}&-\frac{1}{6}\\
\end{bmatrix}.
$$
\end{example}
\subsection{Tensor T-Function}
First, we make some recall for the functions of square matrices based on the Jordan canonical form \cite{Golub1,Higham1}. \par
Let $A\in \mathbb{C}^{n\times n}$ be a matrix with spectrum $\lambda(A):=\{\lambda_j\}_{j=1}^N$, where $N\leq n$ and $\lambda_j$ are distinct. Each $m\times m$ Jordan block $J_m(\lambda)$ of an eigenvalue $\lambda$ has the form
$$
J_m(\lambda)=
\begin{bmatrix}
\lambda & 1 &  &  \\

  & \lambda & \ddots  & \\

  & & \ddots &1\\

 & &  & \lambda
\end{bmatrix}\in \mathbb{C}^{m\times m}.
$$
Suppose that $A$ has the Jordan canonical form
$$
A=XJX^{-1}=X {\rm diag}(J_{m_1}(\lambda_{j_1}),\cdots, J_{m_p}(\lambda_{j_p}))X^{-1},
$$
with $p$ blocks of sizes $m_i$ such that $\sum_{i=1}^{p}m_i=n$, and the eigenvalues $\{\lambda_{j_k}\}_{k=1}^{p}\in {\rm spec}(A)$. \par
\begin{definition}\label{def1-5} {\rm (Matrix function)}
Suppose $A\in \mathbb{C}^{n\times n}$ has the Jordan canonical form and the matrix function is defined as
$$
f(A):=X f(J)X^{-1},
$$
where $f(J):={\rm diag}(f(J_{m_1}(\lambda_{j_1})),\cdots, f(J_{m_p}(\lambda_{j_p})))$, and
$$
f(J_{m_i}(\lambda_{j_i})):=
\begin{bmatrix}
f(\lambda_{j_k}) & f'(\lambda_{j_k}) & \frac{f''(\lambda_{j_k})}{2!} & \cdots & \frac{f^{(n_{j_k}-1)}(\lambda_{j_k})}{(n_{j_k}-1)!} \\

0 & f(\lambda_{j_k}) & f'(\lambda_{j_k}) &  \cdots & \vdots \\

\vdots & \ddots & \ddots &\ddots & \frac{f''(\lambda_{j_k})}{2!} \\

\vdots &   & \ddots & \ddots& f'(\lambda_{j_k}) \\

0 & \cdots & \cdots & 0 & f(\lambda_{j_k})\\
\end{bmatrix}\in \mathbb{C}^{m_i\times m_i}.
$$
\end{definition}
There are various matrix function properties throughout the theorems of matrix analysis
which could be found in the excellent monograph \cite{Higham1}.
By using the concept of T-product, the matrix function is generalized to tensors of size $n\times n \times p$. Suppose we have tensors $\tens{A}\in \mathbb{C}^{n\times n\times p}$ and $\tens{B}\in \mathbb{C}^{n\times s\times p}$, then the tensor T-function of $\tens{A}$ is defined by \cite{Lund1}
$$
f(\tens{A})*\tens{B}:=\fold(f(\bcirc(\tens{A}))\cdot \unfold(\tens{B})),
$$
or equivalently
$$
f(\tens{A}):=\fold(f(\bcirc(\tens{A}))\widehat{E_{1}}^{np\times n}),
$$
here $\widehat{E_{1}}^{np\times n}=\hat{e}_k^p \otimes I_{n\times n}$, where $\hat{e}_k^p\in \mathbb{C}^p$ is the vector of all zeros except for the $k$th entry and $I_{n\times n}$ is the identity matrix, `$\otimes$' is the matrix Kronecker product \cite{Horn1}. \par
There is another way to express $\widehat{E_{1}}^{np\times n}$:
$$
\widehat{E_{1}}^{np\times n}=
\begin{bmatrix}
I_{n\times n}\\
0\\
\vdots\\
0
\end{bmatrix}
=
\begin{bmatrix}
1\\
0\\
\vdots\\
0
\end{bmatrix}\otimes I_{n\times n }=\unfold (\tens{I}_{n\times n \times p}).
$$
Note that $f$ on the right-hand side of the equation is merely the matrix function defined above, so the T-function is well-defined.\par
From this definition, we could see that for a tensor $\mathbb{\tens{A}}\in \mathbb{C}^{n\times n\times p}$, $\bcirc(\tens{A})$ is a block circulant matrix of size $np\times np$. The frontal faces of $\tens{A}$ are the block entries of $A\widehat{E_{1}}^{np\times n}$, then $\tens{A}=\fold(A\widehat{E_{1}}^{np\times n})$, where $A=\bcirc(\tens{A})$.\par
By using the definition of matrix function and Jordan canonical form, Miao, Qi and Wei \cite{3Miao1} introduce the tensor similar relationship and propose the T-Jordan canonical form $\tens{J}$ which is an F-upper-bi-diagonal tensor satisfies
$$
\tens{A}=\tens{P}^{-1}*\tens{J}*\tens{P}.
$$
Then the tensor function can be equivalently defined as
$$
f(\tens{A})=\tens{P}^{-1}*f(\tens{J})*\tens{P}.
$$
The `$\bcirc$' operator satisfies the following relations:
\begin{lemma}\label{lem1-3} {\rm \cite{Lund1}}
Suppose we have tensors $\tens{A}\in \mathbb{C}^{n\times n\times p}$ and $\tens{B}\in \mathbb{C}^{n\times s\times p}$. Then \\
{\rm (1)}\ $\bcirc(\tens{A}*\tens{B})=\bcirc(\tens{A})\bcirc(\tens{B})$,\\
{\rm (2)}\ $\bcirc(\tens{A})^j=\bcirc(\tens{A}^j)$, for all $j=0,1,\ldots$,\\
{\rm (3)}\ $(\tens{A}*\tens{B})^{H}=\tens{B}^H *\tens{A}^H$.\\
{\rm (4)}\ $\bcirc(\tens{A}^{\top})=(\bcirc(\tens{A}))^{\top}$,\ $\bcirc(\tens{A}^{H})=(\bcirc(\tens{A}))^{H}$.
\end{lemma}

\section{Generalized Tensor Functions}
\subsection{Generalized tensor function by T-SVD}
The tensor T-product change problems to block circulant matrices which could be block diagonalizable by the fast Fourier transformation \cite{Chan1,Gleich1}. The calculation of T-product and T-SVD can be done fast, easily and stably because of the following reasons. First, the block circulant operator `bcirc' is only related to the structure of data, which can be constructed in a convenient way. Then the Fast Fourier Transformation and its inverse can be implemented stably, quickly and efficiently. Its algorithm has been fully established. After transform the block circulant matrix into block diagonal matrix, singular value decomposition can be done. The algorithm to compute the singular value decomposition and the compact singular value decomposition is stable and fast. So the computation of the T-SVD and its generalized functions can be done fast, stably, and easily.\par
On the other hand, the tensor T-product and T-functions does have applications in many scientific situations. For example, it can be used in conventional computed
tomography. Semerci, Hao, Kilmer and Miller \cite{Semerci1} introduced the tensor-based formulation and used the ADMM algorithm to solve the TNN model.
They give the quadratic approximation to the Poisson log-likelihood function for $k^{th}$ energy bin as a third order tensor, whose $k^{th}$ frontal slice is given by
$$
L_{k}({\textbf x}_{k})=({\textbf A}{\textbf x}_{k}-{\textbf m}_{k})^{\top}\Sigma_k^{-1}({\textbf A}{\textbf x}_{k}-{\textbf m}_{k}),\quad k=1,2,\ldots, p.
$$
where $\Sigma_k$ is treated as the weighting matrix. To minimize the objective function $L_{k}({\textbf x}_{k})$, it comes to solve the above Least Squares problem, or equivalently, to obtain its T-generalized inverse, i,e., a special case of our generalized functions based on the T-product. They used the T-SVD and compute the T-Least Squares solution. Kilmer and Martin \cite{Kilmer2} also gave the definition of the standard inverse of tensors based on the T-product.\par
Since the tensor T-product and tensor T-function could only be defined for F-square tensors, i.e., tensors of size $n\times n\times p$. In this section, we generalize the concept of tensor T-functions to tensors of size $mn\times n\times p$, that is F-rectangular tensors. In order to do this, we first introduce the T-SVD decomposition of a F-rectangular tensor.
\begin{lemma}\label{lem1-4} {\rm \cite{Hao1,Kilmer1, Kilmer2}}
Let $\tens{A}$ be an $m \times n \times p$ real-valued tensor. Then $\tens{A}$ can be factorized as
$$
\tens{A}=\tens{U} * \tens{S} * \tens{V}^{H},
$$
where $\tens{U}$, $\tens{V}$ are unitary $m \times m \times p$ and $n \times n \times p$ tensor respectively, and $\tens{S}$ is a $m \times n \times p$ F-diagonal tensor.
\end{lemma}
In matrix theories, the generalized matrix function of an $m\times n$ matrix has been introduced by using the matrix singular value decomposition (SVD) \cite{Golub1} and the Moore-Penrose inverse of the matrix \cite{Ben1}.

Let $A\in \mathbb{C}^{m\times n }$ and the singular value decomposition be
$$
A=U\Sigma V^{H}.
$$
Let $r$ be the rank of $A$. Consider the matrices $U_r$ and $V_r$ formed with the first $r$ columns of $U$ and $V$, and let $\Sigma_r$ be the leading $r\times r$ principal submatrix of $\Sigma$ whose diagonal entries are $\sigma_1\geq \sigma_2\geq \cdots \geq \sigma_r>0$. Then we have the compact SVD,
$$
A=U_r \Sigma_r V_r^H.
$$
Hawkins and Ben-Israel \cite{Hawkins1} (or \cite[Chapter 6]{Ben1}) present the spectral theory of rectangular matrices by the SVD.
Let $f:\mathbb{R}\rightarrow \mathbb{R}$ be a scalar function such that $f(\sigma_i)$ is defined for all $i=1,2,\ldots, r$.

Define the generalized matrix function induced by $f$ as
$$
f^{\Diamond}(A)=U_r f(\Sigma_r) V_r^H,
$$
where
$$
f(\Sigma_r)=
\begin{bmatrix}
 f(\sigma_1) &  &  &  \\

  &  f(\sigma_2) &   & \\

  & & \ddots &\\

 & &  &  f(\sigma_r)
\end{bmatrix}.
$$
The induced function $f^{\Diamond}(A)$ reduces to the standard matrix function $f(A)$ whenever $A$ is Hermitian positive definite, or when $A$ is Hermitian positive semi-definite and $f$ satisfies $f(0)=0$.
\begin{lemma}\label{lem1-5} (\cite{Arrigo1}, Proposition $8$) Let $A\in \mathbb{C}^{m\times n}$ be a matrix of rank $r$. Let $f:\mathbb{C}\rightarrow \mathbb{C}$ be a scalar function, and let $f^{\Diamond}: \mathbb{C}^{m\times n}\rightarrow \mathbb{C}^{m\times n}$ be the induced generalized matrix function. Then \\
{\rm (1)}\ $[f^{\Diamond}(A)]^{H}=f^{\Diamond}(A^H)$,\\
{\rm (2)}\ Let $X\in \mathbb{C}^{m\times m}$ and $Y\in \mathbb{C}^{n\times n}$ be two unitary matrices, then $f^{\Diamond}(XAY)=X[f^{\Diamond}(A)]Y$,\\
{\rm (3)}\ If $A=A_1\oplus A_2 \oplus \cdots \oplus A_k$, then $f^{\Diamond}(A)=f^{\Diamond}(A_1)\oplus f^{\Diamond}(A_2) \oplus \cdots \oplus  f^{\Diamond}(A_k) $, where `$\oplus$' is the direct sum of matrices \cite{Horn1}.\\
{\rm (4)}\ $f^{\Diamond}(A)= f\left(\sqrt{A A^H}\right)\left(\sqrt{A A^H}\right)^{\dag} A=A\left(\sqrt{A^H A}\right)^{\dag} f\left(\sqrt{A^H A}\right)$, where $M^{\dag}$ is the Moore-Penrose inverse of $M$ \cite{Ben1}, and $\sqrt{A^H A}$ is the square root of $A^H A$ {\rm \cite[Chapter 6]{Ben1}}.\\
\end{lemma}
Also we introduce these basic results without proof.
\begin{corollary}\label{cor1-1}
Let $A\in \mathbb{C}^{m\times n}$ be a matrix of rank $r$. Let $f:\mathbb{C}\rightarrow \mathbb{C}$ be a scalar function, and let $f^{\Diamond}: \mathbb{C}^{m\times n}\rightarrow \mathbb{C}^{m\times n}$ be the induced generalized matrix function. Then\\
{\rm (1)}\ $[f^{\Diamond}(A)]^{H}=f^{\Diamond}(A^{H})$,\\
{\rm (2)}\ $\overline{f^{\Diamond}(A)}=f^{\Diamond}(\overline{A})$.
\end{corollary}
Now, we  use Lemma \ref{lem1-5} to establish the generalized tensor function for F-rectangular tensors of size $m \times n \times p$.
\begin{theorem}\label{the1-1} {\rm (Generalized tensor function)}
Suppose $\tens{A}\in \mathbb{C}^{m \times n \times p}$ is a third order tensor, $\tens{A}$ has the T-SVD decomposition
$$
\tens{A}=\tens{U} * \tens{S} * \tens{V}^{H},
$$
where $\tens{U}$, $\tens{V}$ are unitary $m \times m \times p$ and $n \times n \times p$ tensors respectively, and $\tens{S}$ is an $m \times n \times p$ F-diagonal tensor, which can be factorized as follows:
\begin{equation}
\bcirc
(\tens{U})=(F_{p}\otimes I_{m} )
\begin{bmatrix}
 U_1 &  &  &  \\

  & U_2 &   & \\

  & & \ddots &\\

 & &  & U_{p}
\end{bmatrix}
(F_{p}^{H}\otimes I_{m} ),\  U_i\in \mathbb{C}^{m\times m}, \ i=1,2, \ldots ,p.
\end{equation}
\begin{equation}
\bcirc
(\tens{S})=(F_{p}\otimes I_{m} )
\begin{bmatrix}
 \Sigma_1 &  &  &  \\

  & \Sigma_2 &   & \\

  & & \ddots &\\

 & &  & \Sigma_{p}
\end{bmatrix}
(F_{p}^{H}\otimes I_{n} ),\  \Sigma_i\in \mathbb{R}^{m\times n}, \ i=1,2, \ldots ,p.
\end{equation}
\begin{equation}
\bcirc
(\tens{V}^{H})=(F_{p}\otimes I_{n} )
\begin{bmatrix}
V_1^{H} &  &  &  \\

  & V_2^{H} &   & \\

  & & \ddots &\\

 & &  & V_{p}^{H}
\end{bmatrix}
(F_{p}^{H}\otimes I_{n} ) ,\  V^{H}_i\in \mathbb{C}^{n\times n}, \ i=1,2, \ldots ,p,
\end{equation}
where $F_{n}$ is the discrete Fourier matrix of size $n\times n$, which is defined as \cite{Chan1}
$$
F_{n\times n}=\frac{1}{\sqrt{n}}
\begin{bmatrix}
 1& 1 &1  &1 & \cdots & 1 \\

 1& \omega & \omega^2 &\omega^3 &\cdots &\omega^{n-1}  \\

 1& \omega^2 & \omega^4 &\omega^6 & \cdots &\omega^{2(n-1)}  \\

 1& \omega^3 & \omega^6 &\omega^9 & \cdots &\omega^{3(n-1)}  \\

 \vdots& \vdots  & \vdots  & \vdots & \ddots&  \vdots \\

1 & \omega^{n-1}  & \omega^{2(n-1)}  &\omega^{3(n-1)} &\cdots & \omega^{(n-1)(n-1)} \\
\end{bmatrix},
$$
where $\omega=e^{-2\pi {\textbf i}/n}$ is the primitive $n$-th root of unity in which ${\textbf i}^2=-1$. $F_{p}^H$ is the conjugate transpose of $F_{p}$.\par
If $f: \mathbb{C}\rightarrow \mathbb{C}$ is a scalar function, then the induced function $f^{\Diamond}: \mathbb{C}^{m \times n \times p} \rightarrow \mathbb{C}^{m \times n \times p}$ could be defined as:
$$
f^{\Diamond}(\tens{A})=\tens{U} * \hat{f}(\tens{S}) * \tens{V}^{H},
$$
where, $\hat{f}(\tens{S})$ is given by
\begin{equation}
\hat{f}(\tens{S})=\bcirc ^{-1} \left(
(F_{p}\otimes I_{m} )
\begin{bmatrix}
 \tilde{f}(\Sigma_1) &  &  &  \\

  &   \tilde{f}(\Sigma_2) &   & \\

  & & \ddots &\\

 & &  &   \tilde{f}(\Sigma_{p})
\end{bmatrix}
(F_{p}^{H}\otimes I_{n} ) \right),
\end{equation}
and
\begin{equation}
\tilde{f}(\Sigma_i)= f\left(\sqrt{\Sigma_i \Sigma_i^H}\right)\left(\sqrt{\Sigma_i \Sigma_i^H}\right)^{\dag} \Sigma_i=\Sigma_i \left(\sqrt{\Sigma_i^H \Sigma_i}\right)^{\dag} f\left(\sqrt{\Sigma_i^H \Sigma_i}\right),\ i=1,2,\ldots, p.
\end{equation}
\end{theorem}
\begin{proof}
Since tensor $\tens{A}$ has the T-SVD decomposition \cite{Kilmer1, Kilmer2}
$$
\tens{A}=\tens{U} * \tens{S} * \tens{V}^{H},
$$
taking ``$\bcirc$" on both sides of the equation
$$
\begin{aligned}
\bcirc(\tens{A})&=\bcirc(\tens{U} * \tens{S} * \tens{V}^{H})\\
&=\bcirc(\tens{U})\cdot \bcirc(\tens{S}) \cdot \bcirc(\tens{V}^{H}),
\end{aligned}
$$
where $\bcirc(\tens{U})$ and $\bcirc(\tens{V}^{H})$ are unitary matrices of order $mp\times mp$ and $np\times np$ respectively, and
$$
\bcirc(\tens{S})=
(F_{p}\otimes I_{m} )
\begin{bmatrix}
 \Sigma_1 &  &  &  \\

  & \Sigma_2 &   & \\

  & & \ddots &\\

 & &  & \Sigma_{p}
\end{bmatrix}
(F_{p}^{H}\otimes I_{n} ),
$$
is a block circulant matrix in $\mathbb{C}^{mp\times np}$.
The induced function on $\tens{S}$ could be defined by
$$
\hat{f}(\tens{S})=\bcirc ^{-1} \left(
(F_{p}\otimes I_{m} )
\begin{bmatrix}
 \tilde{f}(\Sigma_1) &  &  &  \\

  &   \tilde{f}(\Sigma_2) &   & \\

  & & \ddots &\\

 & &  &   \tilde{f}(\Sigma_{p})
\end{bmatrix}
(F_{p}^{H}\otimes I_{n} ) \right),
$$
then we define
$$
\bcirc(f^{\Diamond}(\tens{A}))=\bcirc(\tens{U})\cdot\bcirc(\hat{f}(\tens{S}))\cdot\bcirc(\tens{V}^{H}).
$$
Taking `$\bcirc^{-1}$' on both sides of the equation, it turns out
$$
f^{\Diamond}(\tens{A})=\tens{U}*\hat{f}(\tens{S})*\tens{V}^{H}.
$$
\end{proof}
\begin{remark}\label{rem1-1} An important observation is that the scalar function $f$ could be assumed to be an odd function, since $f$ is only defined for non-negative numbers, we can complete it to an odd function by setting $f(0)=0$ and $f(x)=-f(-x)$ for all $x<0$ \cite{Benzi1}.
\end{remark}
\begin{corollary}\label{cor1-2}
Suppose $\tens{A}\in \mathbb{C}^{m \times n \times p}$ is a third order tensor. Let $f: \mathbb{C}\rightarrow \mathbb{C}$ be a scalar function and let $f^{\Diamond}: \mathbb{C}^{m \times n \times p}\rightarrow \mathbb{C}^{m \times n \times p}$ be the corresponding generalized function of third order tensors. Then \\
{\rm (1)} $ [f^{\Diamond}(\tens{A})]^{\top}=f^{\Diamond}(\tens{A}^{\top})$.\\
{\rm (2)} Let $\tens{P} \in \mathbb{R}^{m \times m \times p}$
and $\tens{Q} \in \mathbb{R}^{n \times n \times p}$ be two orthogonal tensors, then $f^{\Diamond}(\tens{P} *\tens{A}* \tens{Q})=\tens{P}*f^{\Diamond}(\tens{A})* \tens{Q}$.\\
\end{corollary}
\begin{proof}
{\rm (1)} From Theorem \ref{the1-1}, if $\tens{A}$ has the T-SVD decomposition $\tens{A}=\tens{U} * \tens{S} * \tens{V}^{H}$, taking the  transpose on both sides, then it turns to be
$$
\tens{A}^{\top}=\overline{\tens{V}} * \tens{S}^{\top} * \tens{U}^{\top}
$$
it follows that
$$
f^{\Diamond}(\tens{A}^{\top})=\overline{\tens{V}} *\hat{f}( \tens{S})^{\top} * \tens{U}^{\top}=[\tens{U} * \hat{f}( \tens{S}) * \tens{V}^{H}]^{\top}=[f^{\Diamond}(\tens{A})]^{\top}.
$$\\
{\rm (2)} The result follows from the fact that the unitary tensors form a group under the multiplication `$*$'. Define $\tens{B}=\tens{P} *\tens{A}* \tens{Q}$, then we obtain
$$
\begin{aligned}
f^{\Diamond}(\tens{B})&=f^{\Diamond}(\tens{P}*\tens{U}*\tens{S}*\tens{V}^{H}*\tens{Q})\\
&=(\tens{P}*\tens{U})*\hat{f}(\tens{S})*(\tens{Q}^{\top}*\tens{V})^{H}\\
&=\tens{P}*\tens{U}*\hat{f}(\tens{S})*\tens{V}^{H}*\tens{Q}\\
&=\tens{P} *f^{\Diamond}(\tens{A})* \tens{Q}.
\end{aligned}
$$
\end{proof}
\begin{corollary}\label{cor1-3}{\rm(Composite functions)} Suppose $\tens{A}\in \mathbb{C}^{m \times n \times p}$ is a third order tensor and it has T-SVD decomposition $\tens{A}=\tens{U} * \tens{S} * \tens{V}^{H}$. Let $\{\sigma_1, \sigma_2,\cdots, \sigma_r\}$ be the singular values set of $\bcirc (\tens{S})$. Assume $h: \mathbb{C}\rightarrow \mathbb{C}$ and $g: \mathbb{C}\rightarrow \mathbb{C}$ are two scalar functions and $g(h(\sigma_i))$ exists for all $i$.

Let $g^{\Diamond}: \mathbb{C}^{m \times n \times p} \rightarrow \mathbb{C}^{m \times n \times p}$ and $h^{\Diamond}: \mathbb{C}^{m \times n \times p} \rightarrow \mathbb{C}^{m \times n \times p}$ be the induced generalized tensor functions. Moreover, let $f: \mathbb{C}\rightarrow \mathbb{C}$ be the composite function $f=g\circ h$. Then the induced tensor function $f^{\Diamond}: \mathbb{C}^{m \times n \times p} \rightarrow \mathbb{C}^{m \times n \times p}$ satisfies
$$
f^{\Diamond}(\tens{A})=g^{\Diamond}(h^{\Diamond}(\tens{A})).
$$
\end{corollary}
\begin{proof}
Let $\tens{B}=h^{\Diamond}(\tens{A})=\tens{U}*\hat{f}(\Sigma)*\tens{V}^{H}=:\tens{U}*\Theta*\tens{V}^{H}$. Let $\tens{P}\in \mathbb{R}^{m\times n\times p}$ be a permutation tensor (see Definition \ref{def1-11}) such that $\tilde{\Theta}=\tens{P}*\Theta* \tens{P}^{\top}$ has diagonal entries ordered in a non-increasing order. Then it follows that tensor $\tens{B}$ is given by
$$
\tens{B}=\widetilde{\tens{U}}*\widetilde{\Theta}*\widetilde{\tens{V}}^{H},
$$
where $\widetilde{\tens{U}}=\tens{U}*\tens{P}$ and $\widetilde{\tens{V}}=\tens{V}*\tens{P}$. It follows that
$$
\begin{aligned}
g^{\Diamond}(h^{\Diamond}(\tens{A}))&=g^{\Diamond}(\tens{B})\\
&=\widetilde{\tens{U}}*\hat{g}(\widetilde{\Theta})*\widetilde{\tens{V}}^{H}\\
&=\tens{U}*\tens{P}*\hat{g}(\widehat{\Theta})*\tens{P}^{\top}*\tens{V}^H\\
&=\tens{U}*\hat{g}(\tens{P}*\widetilde{\Theta}*\tens{P}^{\top})*\tens{V}^H\\
&=\tens{U}*\hat{g}(\Theta)*\tens{V}^{H}\\
&=\tens{U}*\hat{g}(\hat{h}(\Sigma))*\tens{V}^{H}\\
&=\tens{U}*\hat{f}(\Sigma)*\tens{V}^{H}\\
&=f^{\Diamond}(\tens{A}).
\end{aligned}
$$
\end{proof}
The next corollaries describe the relationship between the generalized tensor functions and the tensor T-functions.
\begin{corollary}\label{cor1-4}
Suppose $\tens{A}\in \mathbb{C}^{m \times n \times p}$ is a third order tensor, and let $f: \mathbb{C}\rightarrow \mathbb{C}$ be a scalar function, and we also use $f$ to denote the tensor T-function. Let $f^{\Diamond}: \mathbb{C}^{m \times n \times p} \rightarrow \mathbb{C}^{m \times n \times p}$ be the induced generalized tensor function. Then
$$
f^{\Diamond}(\tens{A})= f(\sqrt{\tens{A}*\tens{A}^{H}})*(\sqrt{\tens{A}*\tens{A}^{H}})^{\dag} *\tens{A}=\tens{A}*(\sqrt{\tens{A}^{H}*\tens{A}})^{\dag} *f(\sqrt{\tens{A}^{H}*\tens{A}}).
$$
\end{corollary}
\begin{proof}
This identity is the consequence of the fact that the generalized tensor function is equal to the tensor T-function, when the tensor is a F-square tensor.
\end{proof}

\begin{corollary}\label{cor1-5}
Suppose $\tens{A}\in \mathbb{C}^{m \times n \times p}$ is a third order tensor, and let $f: \mathbb{C}\rightarrow \mathbb{C}$ and $g: \mathbb{C}\rightarrow \mathbb{C}$ be two scalar functions, such that $f^{\Diamond}(\tens{A})$ and $g(\tens{A}*\tens{A}^{H})$ are defined. Then we have
$$
g(\tens{A}*\tens{A}^{H})*f^{\Diamond}(\tens{A})=f^{\Diamond}(\tens{A})*g(\tens{A}^{H}*\tens{A}).
$$
\end{corollary}
\begin{proof}
From $\tens{A}=\tens{U} * \tens{S} * \tens{V}^{H}$, it follows $\tens{A}*\tens{A}^{H}=\tens{U}*(\tens{S}*\tens{S}^{H})*\tens{U}^{H}$. Thus we have
$$
\begin{aligned}
g(\tens{A}*\tens{A}^{H})*f^{\Diamond}(\tens{A})&=\tens{U}*g(\tens{S}*\tens{S}^{H})*\tens{U}^{H}*\tens{U} *\hat{f}( \tens{S}) * \tens{V}^{H}\\
&=\tens{U}*\hat{g}(\tens{S}*\tens{S}^{H})*\hat{f}( \tens{S}) * \tens{V}^{H}\\
&=\tens{U}*\hat{f}( \tens{S})*\hat{g}(\tens{S}^{H}*\tens{S}) * \tens{V}^{H}\\
&=\tens{U}*\hat{f}( \tens{S})*\tens{V}^{H}*\tens{V} *\hat{g}(\tens{S}^{H}*\tens{S}) * \tens{V}^{H}\\
&=f^{\Diamond}(\tens{A})*g(\tens{A}^{H}*\tens{A}).
\end{aligned}
$$
\end{proof}
To illustrate the difference between the generalized tensor function and the standard tensor function \cite{Lund1}, we introduce the following example.
\begin{example}
Suppose $\tens{A}\in \mathbb{C}^{1\times 1\times 4}$ is a real tensor and $f:\mathbb{C}\rightarrow\mathbb{C}$ is a scalar function defined as $f(x)=x^2$. We now give the difference result of the tensor function $f(\tens{A})$ and $f^{\Diamond}(\tens{A})$. The elements of $\tens{A}$ is given by
$$
A^{(1)}=1, \ A^{(2)}=2, \ A^{(3)}=3, \ A^{(4)}=4.
$$
then the standard function of $\tens{A}$ is computed as
$$
\begin{aligned}
\bcirc(f(\tens{A}))&=\bcirc(\tens{A}*\tens{A})=\bcirc(\tens{A})\bcirc(\tens{A})\\
&=\begin{bmatrix}
1&2&3&4\\
4&1&2&3\\
3&4&1&2\\
2&3&4&1
\end{bmatrix}^2=\begin{bmatrix}
26&28&26&20\\
20&26&28&26\\
26&20&26&28\\
28&26&20&26
\end{bmatrix}.
\end{aligned}
$$
So the elements of $f(\tens{A})\in \mathbb{R}^{1\times 1\times 4}$ is
$$f(\tens{A})^{(1)}=26,\ f(\tens{A})^{(2)}=20,\ f(\tens{A})^{(3)}=26,\ f(\tens{A})^{(4)}=28.$$
On the other hand, $\bcirc(\tens{A})$ has the singular value decomposition after the discrete Fourier transformation:
$$
\bcirc(\tens{A})=(F_4\otimes I_1)
 \begin{bmatrix}
10&0&0&0\\
0&-2-2{\textbf i}&0&0\\
0&0&-2&0\\
0&0&0&-2+2{\textbf i}\\
\end{bmatrix}(F_4^H\otimes I_1).
$$
The singular value decomposition of the middle matrix is
$$
\begin{aligned}
\begin{bmatrix}
-1&0&0&0\\
0&0&-\frac{1+{\textbf i}}{\sqrt{2}}&0\\
0&0&0&-1\\
0&-\frac{1-{\textbf i}}{\sqrt{2}}&0&0\\
\end{bmatrix}\cdot
\begin{bmatrix}
10&0&0&0\\
0&2\sqrt{2}&0&0\\
0&0&2\sqrt{2}&0\\
0&0&0&2\\
\end{bmatrix}\cdot
\begin{bmatrix}
-1&0&0&0\\
0&0&1&0\\
0&0&0&1\\
0&1&0&0\\
\end{bmatrix}^{\top},
\end{aligned}
$$
so $\bcirc(f^{\Diamond}(\tens{A}))$ equals to
$$
\begin{aligned}
(F_4\otimes I_1)\cdot\begin{bmatrix}
-1&0&0&0\\
0&0&-\frac{1+{\textbf i}}{\sqrt{2}}&0\\
0&0&0&-1\\
0&-\frac{1-{\textbf i}}{\sqrt{2}}&0&0\\
\end{bmatrix}\cdot\begin{bmatrix}
10^2&0&0&0\\
0&(2\sqrt{2})^2&0&0\\
0&0&(2\sqrt{2})^2&0\\
0&0&0&2^2\\
\end{bmatrix}\cdot
\begin{bmatrix}
-1&0&0&0\\
0&0&1&0\\
0&0&0&1\\
0&1&0&0\\
\end{bmatrix}^{\top}\cdot
(F_4^H\otimes I_1)
\end{aligned}
$$
after calculation, we get
$$
\bcirc(f^{\Diamond}(\tens{A}))=
\begin{bmatrix}
24-2\sqrt{2}&26-2\sqrt{2}&24+2\sqrt{2}&26+2\sqrt{2}\\
26+2\sqrt{2}&24-2\sqrt{2}&26-2\sqrt{2}&24+2\sqrt{2}\\
24+2\sqrt{2}&26+2\sqrt{2}&24-2\sqrt{2}&26-2\sqrt{2}\\
26-2\sqrt{2}&24+2\sqrt{2}&26+2\sqrt{2}&24-2\sqrt{2}\
\end{bmatrix}.
$$
That is to say the elements of $f^{\Diamond}(\tens{A})\in \mathbb{R}^{1\times 1\times 4}$ is
$$
f^{\Diamond}(\tens{A})^{(1)}=24-2\sqrt{2},\ f^{\Diamond}(\tens{A})^{(2)}=26+2\sqrt{2},\ f^{\Diamond}(\tens{A})^{(3)}=24+2\sqrt{2},\ f^{\Diamond}(\tens{A})^{(4)}=26-2\sqrt{2}.
$$
So we find the standard tensor function is not equal to the generalized tensor function $f(\tens{A})\neq f^{\Diamond}(\tens{A})$ even though $f(0)=0$ and $\tens{A}$ is a F-square tensor.
\end{example}
The result is if $\tens{A}$ is a normal tensor i.e., $\tens{A}*\tens{A}^H=\tens{A}^H*\tens{A}$, and $f(0)=0$, then we will have $f(\tens{A})= f^{\Diamond}(\tens{A})$.

\subsection{Tensor T-CSVD and orthogonal projection}
The compact singular value decomposition (CSVD) of matrices plays an important role in the Moore-Penrose inverse and the least squares problem since  the compact version of SVD can save the space of storage when the rank of the matrix is not very large and also speed up the running of the algorithm. \par
In this subsection, we extend the T-SVD theorem to tensor T-CSVD based on the T-product of tensors. We first introduce some concept of tensors analogue to the matrix analysis.
\begin{definition}\label{def1-6}
Let $\tens{A}$ be an $m \times n \times p$ real-valued tensor.\\
{\rm (1)} \ The T-range space of $\tens{A}$, $\tens{R}(\tens{A}):= {\rm Ran}((F_p\otimes I_m)\bcirc(\tens{A})(F_p^{H}\otimes I_m))$,\ `{\rm Ran}' means the range space,\\
{\rm (2)} \ The T-null space of $\tens{A}$, $\tens{N}(\tens{A}):={\rm Null}((F_p\otimes I_m)\bcirc(\tens{A})(F_p^{H}\otimes I_m))$, \ `{\rm Null}' represents the null space,\\
{\rm (3)} \ The tensor norm $\norm{\tens{A}}:=\norm{\bcirc(\tens{A})}$,\\
{\rm (4)} \ The Moore-Penrose inverse $\tens{A}^{\dag}=\bcirc^{-1}((\bcirc(\tens{A}))^{\dag})$.
\end{definition}
Now we introduce the tensor T-CSVD based on the T-product according to the T-SVD of tensors.\par
Suppose $\tens{A}$ is an $m \times n \times p$ real-valued tensor. Then we have the tensor T-SVD of $\tens{A}$
$$
\tens{A}=\tens{U} * \tens{S} * \tens{V}^{H},
$$
where $\tens{U}$, $\tens{V}$ are unitary $m \times m \times p$ and $n \times n \times p$ tensors respectively, and $\tens{S}$ is an $m \times n \times p$ F-diagonal tensor, which can be factorized as Equations (1), (2), and (3).\par
Suppose $\Sigma_{i}$ has rank $r_i$, denote $U_i=(u_1^{(i)},u_2^{(i)},\cdots,u_{m}^{(i)})$ and $V_i=(v_1^{(i)},v_2^{(i)},\cdots,v_{n}^{(i)})$, $i=1,2,\ldots, p$ and $r=\max\limits_{1 \leq i \leq p} \{r_i\}$ is usually called the $tubal$-$rank$ of $\tens{A}$ based on the T-product \cite{Kilmer2}. We introduce the T-CSVD by ignoring the ``$0$'' singular values and the corresponding tubes. That is
$$
(\Sigma_i)_{r}={\rm diag}(c_1^{(i)},c_2^{(i)},\cdots,c_{r}^{(i)})\in \mathbb{R}^{r\times r},
$$
$$
(U_i)_{r}=(u_1^{(i)},u_2^{(i)},\cdots,u_{r}^{(i)})\in \mathbb{C}^{m\times r},
$$
$$
(V_i)_{r}=(v_1^{(i)},v_2^{(i)},\cdots,v_{r}^{(i)})\in \mathbb{C}^{n\times r},
$$
then it turns out to be
$$
\begin{aligned}
\bcirc(\tens{A})=&(F_{p}\otimes I_{m} )
\begin{bmatrix}
 (U_1)_{r} &  & &\\
& (U_2)_{r}   & &\\
 & & \ddots &    \\

 & & & (U_{p})_{r} \\

\end{bmatrix}
(F_{p}^{H}\otimes I_r )\\
&\times(F_{p}\otimes I_r )
\begin{bmatrix}
 (\Sigma_1)_{r} &  & \\
&  (\Sigma_2)_{r}   & &\\
&  & \ddots &    \\

&  & & (\Sigma_{p})_{r} \\
\end{bmatrix}(F_{p}^{H}\otimes I_r )\\
&\times(F_{p}\otimes I_r )
\begin{bmatrix}
 (V_1)_{r}^H &  & &\\
& (V_2)_{r}^H   & &\\
 & & \ddots &    \\

 & & & (V_{p})_{r}^H \\

\end{bmatrix}
(F_{p}^{H}\otimes I_{n} )\\
=&\bcirc(\tens{U}_{(r)})\bcirc(\tens{S}_{(r)})\bcirc(\tens{V}_{(r)}^{H}),
\end{aligned}
$$
where $\tens{U}_{(r)}\in \mathbb{C}^{m\times r\times p}$, $\tens{S}_{(r)}\in \mathbb{C}^{r\times r\times p}$, $\tens{V}_{(r)}^{H}\in \mathbb{C}^{r\times n\times p}$,
that is to say
\begin{equation}
\tens{A}=\tens{U}_{(r)}*\tens{S}_{(r)}*\tens{V}_{(r)}^{H}.
\end{equation}
\begin{remark}\label{rem1-2}
In matrix analysis, if a matrix $A$ has CSVD
$$
A=U_r S_r V_r^H,
$$
where $S_r={\rm diag}(\sigma_1,\sigma_2 \cdots, \sigma_r)$, we have $\sigma_i\neq0$, for all $i\in\{1,2,\ldots,r\}$, but in T-CSVD for tensors, we may have
$c^{(i)}_j= 0$ for some $i$ and $j$, since we choose $r=\max\limits_{i}\{r_i\}$.
\end{remark}
We give an example to illustrate our T-CSVD.
\begin{example}Suppose $\tens{A}\in \mathbb{C}^{3\times 3\times 3}$ is a complex tensor whose elements of the frontal slices are given as:
$$
A^{(1)}=\begin{bmatrix}
\frac{2}{3}&0&0\\
0&\frac{5}{3}&0\\
0&0&\frac{2}{3}\\
\end{bmatrix},\quad
A^{(2)}=\begin{bmatrix}
\frac{1}{6}+\frac{\sqrt{3}}{6}\textbf{i}&0&0\\
0&-\frac{5}{6}-\frac{\sqrt{3}}{6}\textbf{i}&0\\
0&0&-\frac{1}{3}-\frac{\sqrt{3}}{3}\textbf{i}\\
\end{bmatrix},
$$
$$
A^{(3)}=\begin{bmatrix}
\frac{1}{6}-\frac{\sqrt{3}}{6}\textbf{i}&0&0\\
0&-\frac{5}{6}+\frac{\sqrt{3}}{6}\textbf{i}&0\\
0&0&-\frac{1}{3}+\frac{\sqrt{3}}{3}\textbf{i}\\
\end{bmatrix}.
$$
Then we have
$$
\bcirc(\tens{A})=(F_{3}\otimes I_{3} )
\begin{bmatrix}
 A_1 &  &\\
& A_2   &\\
&&A_3
\end{bmatrix}
(F_{3}^{H}\otimes I_3 ),
$$
where
$$
A_1=\begin{bmatrix}
1&0&0\\
0&0&0\\
0&0&0\\
\end{bmatrix},\quad
A_2=\begin{bmatrix}
1&0&0\\
0&2&0\\
0&0&0\\
\end{bmatrix},\quad
A_3=\begin{bmatrix}
0&0&0\\
0&3&0\\
0&0&2\\
\end{bmatrix}.
$$
So the tubal-rank of $\tens{A}$ is $r_t(\tens{A})=2$. Each $A_i,\ i=1,2,3$ can be decomposite as follows:
$$
A_1=(U_1)_r(\Sigma_1)_r(V_1)_r^{H}=\begin{bmatrix}
1&0\\
0&0\\
0&1\\
\end{bmatrix}
\begin{bmatrix}
1&0\\
0&0\\
\end{bmatrix}
\begin{bmatrix}
1&0&0\\
0&0&1\\
\end{bmatrix},
$$
$$
A_2=(U_2)_r(\Sigma_2)_r(V_2)_r^{H}=\begin{bmatrix}
0&1\\
1&0\\
0&0\\
\end{bmatrix}
\begin{bmatrix}
2&0\\
0&1\\
\end{bmatrix}
\begin{bmatrix}
0&1&0\\
1&0&0\\
\end{bmatrix},
$$
$$
A_3=(U_3)_r(\Sigma_3)_r(V_3)_r^{H}=\begin{bmatrix}
0&0\\
1&0\\
0&1\\
\end{bmatrix}
\begin{bmatrix}
3&0\\
0&2\\
\end{bmatrix}
\begin{bmatrix}
0&1&0\\
0&0&1\\
\end{bmatrix}.
$$
By the inverse Fast Fourier Transformation, the T-CSVD decomposition of $\tens{A}$ is given by
$$
\tens{A}=\tens{U}_{(r)}*\tens{S}_{(r)}*\tens{V}_{(r)}^H,
$$
where the frontal slices of tensor $\tens{U}_{(r)}\in\mathbb{C}^{3\times2\times3},\tens{S}_{(r)}\in\mathbb{C}^{2\times2\times3},\tens{V}_{(r)}^H\in\mathbb{C}^{2\times3\times3}$ are given by
$$
U^{(1)}_{(r)}=\frac{1}{3}\begin{bmatrix}
1&1\\
2&0\\
0&2\\
\end{bmatrix},\quad
U^{(2)}_{(r)}=\frac{1}{3}\begin{bmatrix}
1&-\frac{1}{2}+\frac{\sqrt{3}}{2}\textbf{i}\\
-1&0\\
0&-\frac{1}{2}-\frac{\sqrt{3}}{2}\textbf{i}\\
\end{bmatrix},\quad
U^{(3)}_{(r)}=\frac{1}{3}\begin{bmatrix}
1&-\frac{1}{2}-\frac{\sqrt{3}}{2}\textbf{i}\\
-1&0\\
0&-\frac{1}{2}+\frac{\sqrt{3}}{2}\textbf{i}\\
\end{bmatrix},
$$
$$
S^{(1)}_{(r)}=\frac{1}{3}
\begin{bmatrix}
6&0\\
0&3\\
\end{bmatrix},\quad
S^{(2)}_{(r)}=\frac{1}{3}
\begin{bmatrix}
-\frac{3}{2}-\frac{\sqrt{3}}{2}\textbf{i}&0\\
0&-\frac{3}{2}-\frac{\sqrt{3}}{2}\textbf{i}\\
\end{bmatrix},\quad
S^{(3)}_{(r)}=\frac{1}{3}
\begin{bmatrix}
-\frac{3}{2}+\frac{\sqrt{3}}{2}\textbf{i}&0\\
0&-\frac{3}{2}+\frac{\sqrt{3}}{2}\textbf{i}\\
\end{bmatrix},
$$
$$
V^{H(1)}_{(r)}=\frac{1}{3}\begin{bmatrix}
1&2&0\\
1&0&2\\
\end{bmatrix},\quad
V^{H(2)}_{(r)}=\frac{1}{3}\begin{bmatrix}
1&-1&0\\
-\frac{1}{2}+\frac{\sqrt{3}}{2}\textbf{i}&0&\frac{1}{2}-\frac{\sqrt{3}}{2}\textbf{i}\\
\end{bmatrix},
$$
$$
V^{H(3)}_{(r)}=\frac{1}{3}\begin{bmatrix}
1&-1&0\\
-\frac{1}{2}-\frac{\sqrt{3}}{2}\textbf{i}&0&\frac{1}{2}+\frac{\sqrt{3}}{2}\textbf{i}\\
\end{bmatrix}.
$$
\end{example}
\begin{corollary}\label{cor1-6}
Suppose $\tens{A}$ is an $m \times n \times p$ real-valued tensor. Then we have the tensor T-SVD of $\tens{A}$ (see Fig. \ref{fig1-1}) is
$$
\tens{A}=\tens{U} * \tens{S} * \tens{V}^{H},
$$
the T-CSVD of $\tens{A}$ is
$$
\tens{A}=\tens{U}_{(r)} * \tens{S}_{(r)} * \tens{V}^{H}_{(r)},
$$
and the Moore-Penrose inverse of tensor $\tens{A}$ is given by
$$
\tens{A}^{\dag}=\tens{V}_{(r)} * \tens{S}_{(r)}^{\dag} * \tens{U}_{(r)}^{H}.
$$

\end{corollary}
\begin{proof}
It is obvious by using the identity
$$
\bcirc(\tens{A}^{\dag})=(\bcirc(\tens{A}))^{\dag},
$$
where $(\bcirc(\tens{A}))^{\dag}$ is defined by matrix Moore-Penrose inverse \cite{Ben1}.
\end{proof}
It is easy to establish the following results of projection.
\begin{corollary}\label{cor1-7}
Suppose $\tens{A}$ is an $m \times n \times p$ real-valued tensor and the T-CSVD of $\tens{A}$ is
$$
\tens{A}=\tens{U}_{(r)} * \tens{S}_{(r)} * \tens{V}^{H}_{(r)},
$$
then we have\\
{\rm (1)} \ $\bcirc(\tens{U}_{(r)} *\tens{U}_{(r)}^{H})={\rm P}_{\tens{R}(\tens{A})}=\bcirc(\tens{A}*\tens{A}^{\dag})$, $\tens{U}_{(r)}^{H}*\tens{U}_{(r)}=\tens{I}_{r} $,\\
{\rm (2)} \ $\bcirc(\tens{V}_{(r)} *\tens{V}_{(r)}^{H})={\rm P}_{\tens{R}(\tens{A}^{H})}=\bcirc(\tens{A}^{\dag}*\tens{A})$, $\tens{V}_{(r)}^{H}*\tens{V}_{(r)}=\tens{I}_{r} $,\\
{\rm (3)} \ The tensor $\tens{E}:=\tens{U}_{(r)}*\tens{V}_{(r)}^{H}$ is a real partial isometry tensor satisfying $\bcirc(\tens{E}*\tens{E}^{H})={\rm P}_{\tens{R}(\tens{A})}$ and $\bcirc(\tens{E}^{H}*\tens{E})={\rm P}_{\tens{R}(\tens{A}^{H})}$.
\end{corollary}
 \begin{figure}\label{fig1-3}
  \centering
  \includegraphics[width=1\textwidth]{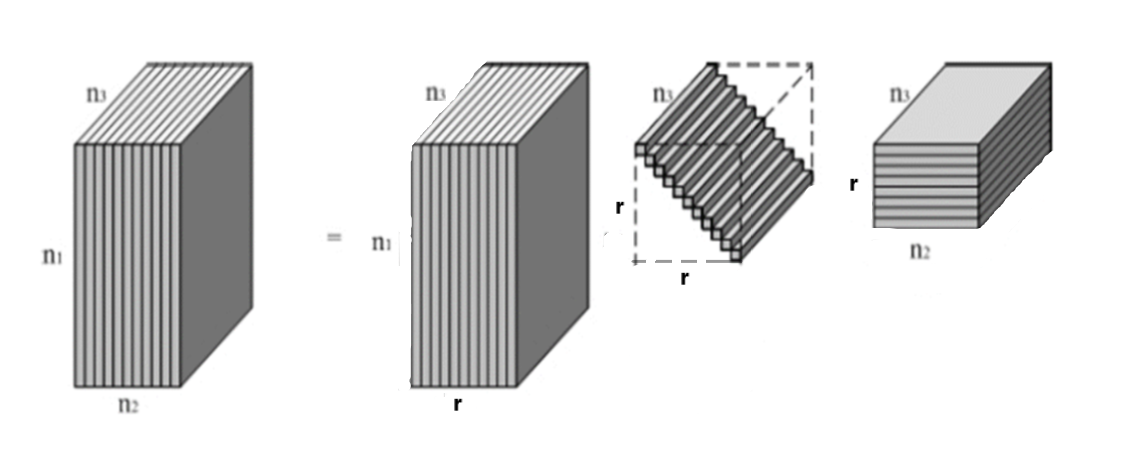}
  \caption{T-CSVD of Tensors}
\end{figure}
By using the Tensor T-CSVD, we can also derive the generalized tensor function of tensors. For simplicity, we give the following results without proof.
\begin{theorem}\label{the1-2} {\rm (Compact tensor singular value decomposition)}
Suppose $\tens{A}$ is an $m \times n \times p$ real-valued tensor and the T-CSVD of $\tens{A}$ is
$$
\tens{A}=\tens{U}_{(r)} * \tens{S}_{(r)} * \tens{V}^{H}_{(r)}.
$$
If $f: \mathbb{C}\rightarrow \mathbb{C}$ is a scalar function, then the induced function $f^{\Diamond}: \mathbb{C}^{m \times n \times p} \rightarrow \mathbb{C}^{m \times n \times p}$ could be defined as:
$$
f^{\Diamond}(\tens{A})=\tens{U}_{(r)} * \hat{f}(\tens{S}_{(r)}) * \tens{V}_{(r)}^{H},
$$
where $\hat{f}(\tens{S}_{(r)})$ is defined as above.
\end{theorem}
\begin{remark}\label{rem1-3} From the above illustration, since $\tens{S}_{(r)}$ is a special tensor, we find that $\hat{f}(\tens{S}_{(r)})=f^{\Diamond}(\tens{S}_{(r)})$. Thus, $f^{\Diamond}(\tens{A})=\tens{U}_{(r)} * f^{\Diamond}(\tens{S}_{(r)}) * \tens{V}_{(r)}^{H}$.
\end{remark}
By using the projection tensor $\tens{E}$ of the T-CSVD of a tensor, we obtain the following results.
\begin{corollary}\label{cor1-8}
If $f, g, h: \mathbb{C}\rightarrow \mathbb{C}$ are scalar functions, and $f^{\Diamond}, g^{\Diamond}, h^{\Diamond}: \mathbb{C}^{m \times n \times p} \rightarrow \mathbb{C}^{m \times n \times p}$ are the corresponding induced generalized tensor functions. Suppose $\tens{A}$ is an $m \times n \times p$ real-valued tensor and the T-CSVD of $\tens{A}$ is
$$
\tens{A}=\tens{U}_{(r)} * \tens{S}_{(r)} * \tens{V}^{H}_{(r)},
$$
then the partial isometry tensor $\tens{E}=\tens{U}_{(r)}*\tens{V}_{(r)}^{H}$ is a real tensor,
and we have\\
{\rm (1)} \ If $f(z)=k$, then $f^{\Diamond}(\tens{A})=k\tens{E}$,\\
{\rm (2)} \ If $f(z)=z$, then $f^{\Diamond}(\tens{A})=\tens{A}$,\\
{\rm (3)} \ If $f(z)=g(z)+h(z)$, then $f^{\Diamond}(\tens{A})=g^{\Diamond}(\tens{A})+h^{\Diamond}(\tens{A})$,\\
{\rm (3)} \ If $f(z)=g(z)h(z)$, then $f^{\Diamond}(\tens{A})=g^{\Diamond}(\tens{A})*\tens{E}^{H}*h^{\Diamond}(\tens{A})$.
\end{corollary}
\begin{proof}
It follows easily from $\tens{E}^{H}=\tens{E}^{\dag}$.
\end{proof}

\section{Explicit Formula for Generalized Tensor Functions}
\subsection{T-Eigenvalue and Cauchy integral theorem of tensors}
In functional analysis and matrix theories, the partial isometries in Hilbert spaces were studied extensively by von Neumann, Halmos, Erdelyi, and others. A linear operator $U: \mathbb{C}^{n}\rightarrow \mathbb{C}^{m}$ is called a partial isometry if it is norm preserving on the orthogonal complement of its null space, i.e., if
$$
\norm{Ux}=\norm{x},\quad {\rm for \ all}\ x\in \tens{N}(U)^{\perp}=\tens{R}(U^H),
$$
or, equivalently, if it is distance preserving
$$
\norm{Ux-Uy}=\norm{x-y},\quad {\rm for \ all\ } x,y\in\tens{N}(U)^{\perp}.
$$
For matrix cases we have the spectral theorem for rectangular matrices \cite{Ben1}.\par
Suppose $O\neq A\in\mathbb{C}^{m\times n}$ of rank $r$. Let $d(A)=\{d_1,d_2\ldots, d_r\}$ be complex scalars satisfying
$$
|d_i|=\sigma_i,\quad i=1,2,\ldots,r,
$$
where
$$
\sigma_1\geq \sigma_2\geq \cdots\geq\sigma_r>0
$$
are the singular values, i.e., $\sigma(A)$, of $A$.\par
Then there exist $r$ partial isometries $E_i\in\mathbb{C}_{m\times n},\ i=1,2,\ldots,r$ which are all rank $1$ matrices satisfying
$$
E_iE_j^H=O,  \quad E_i^HE_j=O,\quad 1\leq i\neq j\leq r,
$$
$$
E_iE^HA=AE^HE_i,\quad i=1,2,\ldots,r,
$$
where
$$
E=\sum_{i=1}^{r}E_i
$$
is the partial isometry and
$$
A=\sum_{i=1}^rd_iE_i.
$$\par
For third order tensors based on the T-product, we can also introduce $r\times p$ partial isometries tensors of a non-zero tensor $\tens{A}$ which need to satisfy the following equations:
$$
 \tens{E}_j^{(i)}*\tens{E}_l^{(k)H}=0 , \tens{E}_j^{(i)H}*\tens{E}_l^{(k)}=0,\ {\rm for} \  i\neq k \ {\rm or}\  j\neq l.
$$
$$
 \tens{E}_j^{(i)}* \tens{E}^{H}*\tens{A}= \tens{A}* \tens{E}^{H}*\tens{E}_j^{(i)}.
$$
By the definition of $\tens{E}$, we have
$$
\bcirc(\tens{E})=(F_{p}\otimes I_{m} )
\begin{bmatrix}
 (U_1)_{r} &  & &\\
& (U_2)_{r}   & &\\
 & & \ddots &    \\

 & & & (U_{p})_{r} \\

\end{bmatrix}
\begin{bmatrix}
 (V_1)_{r}^H &  & &\\
& (V_2)_{r}^H   & &\\
 & & \ddots &    \\

 & & & (V_{p})_{r}^H \\

\end{bmatrix}
(F_{p}^{H}\otimes I_{n} ),
$$
then we define the complex tensor $\tens{E}_j^{(i)}$ by
$$
\bcirc(\tens{E}_j^{(i)})=(F_{p}^{H} \otimes I_{m})(u_j^{(i)} v_j^{(i)H}) (F_{p}\otimes I_{n} ),\ i=1,2,\ldots ,p,\ j=1,2,\ldots, r.
$$
and it is easy to find
$$
\tens{E}=\sum_{i,j}\tens{E}_j^{(i)},
$$
which means the projection tensor $\tens{E}$ is the sum of the partial isometry tensors $\tens{E}_j^{(i)}$, i.e., $(\tens{E}_j^{(i) })^\dag=(\tens{E}_j^{(i) })^H$. \par
It should be noticed that although $u_j^{(i)}$ and $(v_j^{(i)})^H$ are complex vectors, the isometry matrix $\bcirc(\tens{E})$ can be proved to be a real bicirculant matrix (without proof), and this is the main reason why $f^{\Diamond}$ maps a real tensor to a real tensor, instead of a complex tensor.\par
\begin{remark}\label{rem1-4} It should be noticed that $\tens{E}_j^{(i)}$ is the partial isometry tensor and the superscript `$(i)$' and subscript `$j$' is not related to the frontal slices and the block matrix after Fast Fourier Transformation.
\end{remark}
Suppose $\tens{A}$ is an $m \times n \times p$ real-valued tensor and the T-CSVD of $\tens{A}$ is
$$
\tens{A}=\tens{U}_{(r)} * \tens{S}_{(r)} * \tens{V}^{H}_{(r)},
$$
the set of T-eigenvalues of $\tens{A}$ is defined to be the set ${\rm spec}(\bcirc(\tens{S}_{(r)}*\tens{S}_{(r)}^{H}))$, which is equal to the set of all the positive eigenvalues of the matrix $\bcirc(\tens{A}*\tens{A}^{H})$.\par
Denote the T-eigenvalue set ${\rm spec}(\bcirc(\tens{S}_{(r)}*\tens{S}_{(r)}^{H}))=\{|c_j^{(i)}|^2, 1\leq i\leq p, 1\leq j\leq r \}$, where $r=\sum_{i=1}^{p}r_i$. We call the non-zero elements of $c_j^{(i)}$ the T-singular values of $\tens{A}$.  It can be found that
$$
\tens{A}=\sum_{i,j}c_j^{(i)} \tens{E}_j^{(i)}.
$$
We use the set $\{\gamma_k^{(l)}\}$ to be the set of distinct $c_j^{(i)}$'s. \par
In the following theorems we need our function $f$ to satisfy the following conditions: suppose $f: \mathbb{C}\rightarrow \mathbb{C}$ is a scalar function, $\Gamma_k^{(l)}$ is a simple closed positively oriented contour such that\par
 {\rm (1)} If there is some $c_j^{(i)}=0$, then we need $f$ to be analytic on and  inside $\Gamma_k^{(l)}$ and $f(0)=0$. If $c_j^{(i)}\neq0$ for all $i,j$, we only need $f$ to be analytic on and  inside $\Gamma_k^{(l)}$. \par
{\rm (2)} $\gamma_k^{(l)}$ is inside $\Gamma_k^{(l)}$ but no other $\gamma_{k'}^{(l')}$ is on or inside $\Gamma_k^{(l)}$, and the tensor $\bar{\tens{E}}_k^{(l)}$ is defined by
$$
\bar{\tens{E}}_k^{(l)}=\sum_{c_j^{(i)}=\gamma_k^{(l)}} \tens{E}_j^{(i)}.
$$
\begin{theorem}\label{the1-3} Suppose $f: \mathbb{C}\rightarrow \mathbb{C}$ is a scalar function and $\Gamma_k^{(l)}$ is a contour satisfying the above conditions respectively.\\
{\rm (1)}
The relationship between $\bar{\tens{E}}_k^{(l)}$ and $\Gamma_k^{(l)}$ is
\begin{equation}
\bar{\tens{E}}_k^{(l)}= \tens{E}*\left(\frac{1}{2\pi {\textbf i}} \int_{\Gamma_k^{(l)}} (z \tens{E}-\tens{A})^{\dag} {\rm d} z\right)* \tens{E}.
\end{equation}
{\rm (2)} The tensor function $f^{\Diamond}: \mathbb{C}^{m \times n \times p} \rightarrow \mathbb{C}^{m \times n \times p}$ induced by $f: \mathbb{C}\rightarrow \mathbb{C}$ is
\begin{equation}
f^{\Diamond}(\tens{A})=\tens{E}*\left(\frac{1}{2\pi {\textbf i}} \int_{\Gamma} f(z)(z \tens{E}-\tens{A})^{\dag} {\rm d} z\right)* \tens{E},
\end{equation}
where $\Gamma=\bigcup_{k,l}\Gamma_k^{(l)}$.\par
 In particular, if all $c_{j}^{(i)}\neq 0$ for the tensor $\tens{A}$, then
\begin{equation}
\tens{A}^{\dag}=\frac{1}{2\pi {\textbf  i}}\int_{\Gamma} \frac{1}{z} (z \tens{E}-\tens{A})^{\dag} {\rm d}z.
\end{equation}
\end{theorem}
\begin{proof}
{\rm (1)} From the T-CSVD of $\tens{A}$, we have
$$
z \tens{E}-\tens{A}=\tens{U}_{(r)}*(z\tens{I}-\tens{S}_{(r)})*\tens{V}^{H}_{(r)}
$$
taking the Moore-Penrose inverse on both sides of the equation, it comes to
$$
(z \tens{E}-\tens{A})^{\dag}=\tens{V}_{(r)}*(z\tens{I}-\tens{S}_{(r)})^{-1}*\tens{U}^{H}_{(r)}
$$
Let $\tens{W}_k^{(l)}$ be the right-hand side of the equality need to be proved,
$$
\begin{aligned}
\tens{W}_k^{(l)}&=\tens{E}*\left(\frac{1}{2\pi {\textbf i}} \int_{\Gamma_k^{(l)}} (z \tens{E}-\tens{A})^{\dag} {\rm d} z\right)* \tens{E}\\
&=\tens{U}_{(r)}*\tens{V}_{(r)}^{H}*\left(\frac{1}{2\pi {\textbf i}} \int_{\Gamma_k^{(l)}} \tens{V}_{(r)}* (z \tens{I}-\tens{S}_{(r)})^{-1}* \tens{U}_{(r)}^{H} {\rm d} z\right)*\tens{U}_{(r)}*\tens{V}_{(r)}^{H}\\
&=\tens{U}_{(r)}*\left(\frac{1}{2\pi {\textbf i}} \int_{\Gamma_k^{(l)}} (z \tens{I}-\tens{S}_{(r)})^{-1}{\rm d} z\right)*\tens{V}_{(r)}^{H},\\
\end{aligned}
$$
by taking `$\bcirc$' and block fast Fourier transformation on both sides of the above equation, we have
$$
\begin{aligned}
\bcirc(\tens{W}_k^{(l)})&=\bcirc\left(\tens{U}_{(r)}*\left(\frac{1}{2\pi {\textbf i}} \int_{\Gamma_k^{(l)}} (z \tens{I}-\tens{S}_{(r)})^{-1}{\rm d} z\right)*\tens{V}_{(r)}^{H}\right)\\
&=\bcirc(\tens{U}_{(r)})\bcirc\left(\frac{1}{2\pi {\textbf i}} \int_{\Gamma_k^{(l)}} (z \tens{I}-\tens{S}_{(r)})^{-1}{\rm d} z\right)\bcirc(\tens{V}_{(r)}^{H})\\
&=\bcirc(\tens{U}_{(r)})\bcirc\left(\frac{1}{2\pi {\textbf i}} \int_{\Gamma_k^{(l)}} (z-c_{k'}^{(l)'})^{-1}{\rm d} z\right)\bcirc(\tens{V}_{(r)}^{H})\\
&=\bcirc(\tens{U}_{(r)}){\rm diag}(w_{k'}^{(l)'})\bcirc(\tens{V}_{(r)}^{H}),
\end{aligned}
$$
where
 $$
w_{k'}^{(l)'}=\frac{1}{2\pi {\textbf i}}\int_{\Gamma_k^{(l)}} (z-c_{k'}^{(l)'})^{-1}{\rm d} z=
\left\{
\begin{aligned}
0, \ {\rm if } \ c_{k'}^{(l)'}\notin\Gamma_k^{(l)} \\
1, \ {\rm if } \ c_{k'}^{(l)'}\in\Gamma_k^{(l)}, \\
\end{aligned}
\right.
$$
by the definition of $\Gamma=\bigcup_{k,l}\Gamma_k^{(l)}$ and Cauchy integral formula,
$$
\bcirc(\tens{W}_k^{(l)})=\bcirc\left(\sum_{c_{k'}^{(l)'}=\gamma_k^{(l)}}(u_j^{(i)} v_j^{(i)H})\right) = \bcirc(\bar{\tens{E}}_k^{(l)}).
$$
{\rm (2)} Similarly,
$$
\begin{aligned}
&\bcirc\left(\tens{E}*\left(\frac{1}{2\pi {\textbf i}} \int_{\Gamma} f(z)(z \tens{E}-\tens{A})^{\dag} {\rm d} z\right)* \tens{E}\right)\\
=&\bcirc(\tens{U}_{(r)})\bcirc\left({\rm diag}\left(\frac{1}{2\pi {\textbf i}} \int_{\Gamma_k^{(l)}} f(z) (z-c_{k'}^{(l)'})^{-1}{\rm d} z\right)\right)\bcirc(V_{(r)}^{H})\\
=&\sum_{i,j}f(c_{k}^{(l)})\bcirc(\tens{E}_{k}^{(l)})\\
=&\bcirc(f^{\Diamond}(\tens{A})).
\end{aligned}
$$
The final step is because of the assumption, if there is some $c_{k'}^{(l)'}=0$, then we have $f(0)=0$.
\end{proof}
Since the resolvent of matrices and tensors plays an important role in the matrix equations such as $AXB=D$, we generalize the definition of resolvent to tensors based on the tensor T-product. The tensor resolvent satisfies the following important property:
\begin{theorem}\label{the1-4}
The generalized tensor resolvent of $\tens{A}$ is the function $\widehat{\tens{R}}(z,\tens{A})$ given by
\begin{equation}
\widehat{\tens{R}}(z,\tens{A})=(z\tens{E}-\tens{A})^{\dag},
\end{equation}
then
\begin{equation}
\widehat{\tens{R}}(\lambda,\tens{A})-\widehat{\tens{R}}(\mu,\tens{A})=(\mu-\lambda)\widehat{\tens{R}}(\lambda,\tens{A})*\widehat{\tens{R}}(\mu,\tens{A})
\end{equation}
for any $\lambda$, $\mu$ $\neq c_j^{(i)}$.
\end{theorem}
\begin{proof}
From the above proof, we have
$$
(z\tens{E}-\tens{A})^{\dag}=\sum_{i,j}\frac{1}{z-c_j^{(i)}} \tens{E}_j^{(i)H}.
$$
Then the left-hand side of $(11)$ comes to
$$
\begin{aligned}
\widehat{\tens{R}}(\lambda,\tens{A})-\widehat{\tens{R}}(\mu,\tens{A})&=\sum_{i,j}\left(\frac{1}{\lambda-c_j^{(i)}}-\frac{1}{\mu-c_j^{(i)}}\right)\tens{E}_j^{(i)H}\\
&=\sum_{i,j}\left(\frac{\mu-\lambda}{(\lambda-c_j^{(i)})(\mu-c_j^{(i)})}\right)\tens{E}_j^{(i)H}\\
&=(\mu-\lambda)\left(\sum_{i,j}\frac{1}{\lambda-c_j^{(i)}}\tens{E}_j^{(i)H}\right)*\tens{E}*\left(\sum_{k,l}\frac{1}{\mu-c_l^{(k)}}\tens{E}_l^{(k)H}\right)\\
&=(\mu-\lambda)\widehat{\tens{R}}(\lambda,\tens{A})*\widehat{\tens{R}}(\mu,\tens{A}).
\end{aligned}
$$
\end{proof}
We illustrate now the application of the above concepts to the solution of the tensor equation:
$$
\tens{A}*\tens{X}*\tens{B}=\tens{D}.
$$
Here the tensors $\tens{A}\in\mathbb{R}^{m\times n\times p}$, $\tens{B}\in\mathbb{R}^{k\times l\times p}$ and $\tens{D}\in\mathbb{R}^{m\times l\times p}$ and the partial isometries are given by
$$
\tens{A}=\sum_{i,j=1}^{p, r^{\tens{A}}}c_j^{(i)\tens{A}}\tens{E}_j^{(i)\tens{A}}, \ \tens{E}^{\tens{A}}=\sum_{i,j=1}^{p, r^{\tens{A}}}\tens{E}_j^{(i)\tens{A}},
$$
$$
\tens{B}=\sum_{i,j=1}^{p, r^{\tens{B}}}c_j^{(i)\tens{B}}\tens{E}_j^{(i)\tens{B}}, \ \tens{E}^{\tens{B}}=\sum_{i,j=1}^{p, r^{\tens{B}}}\tens{E}_j^{(i)\tens{B}}.
$$
\begin{theorem}\label{the1-5}
Let $\tens{A}$, $\tens{B}$, $\tens{D}$ be as above, and let $\Gamma_1$ and $\Gamma_2$ be contours surrounding $c(\tens{A})=\{c_j^{(i)\tens{A}}, i=1,2,\ldots, p, j=1,2, \ldots, r^{\tens{A}}\}$ and $c(\tens{B})=\{c_j^{(i)\tens{B}}, i=1,2,\ldots, p, j=1,2, \ldots, r^{\tens{B}}\}$, where $r^{\tens{A}}$ and $r^{\tens{B}}$ are the tubal rank of tensor $\tens{A}$ and $\tens{B}$ respectively. Then the above tensor equation has the following solution:
\begin{equation}
\tens{X}=-\frac{1}{4\pi^2}\int_{\Gamma_1}\int_{\Gamma_2}\frac{\widehat{\tens{R}}(\lambda,\tens{A})*\tens{D}*\widehat{\tens{R}}(\mu,\tens{B})}{\lambda\mu}\  {\rm d}\mu {\rm d}\lambda.
\end{equation}
\end{theorem}
\begin{proof}
It follows from Theorem \ref{the1-3} that
$$
\tens{A}=\tens{E}^{\tens{A}}*\left(\frac{1}{2 \pi {\textbf i}}\int_{\Gamma_1}\lambda \widehat{\tens{R}}(\lambda,\tens{A})\ {\rm d}\lambda\right)*\tens{E}^{\tens{A}},
$$
$$
\tens{B}=\tens{E}^{\tens{B}}*\left(\frac{1}{2 \pi {\textbf i}}\int_{\Gamma_2}\mu \widehat{\tens{R}}(\mu,\tens{B})\ {\rm d}\mu\right)*\tens{E}^{\tens{B}}.
$$
Therefore,
$$
\begin{aligned}
\tens{A}*\tens{X}*\tens{B}&=\tens{A}*\left( \frac{1}{2 \pi {\textbf i}} \int_{\Gamma_1} \frac{\widehat{\tens{R}}(\lambda,\tens{A})}{\lambda}*\tens{D}*\left(  \frac{1}{2 \pi {\textbf i}} \int_{\Gamma_2} \frac{\widehat{\tens{R}}(\mu,\tens{A})}{\mu}\ {\rm d}\mu \right)\ {\rm d}\lambda   \right)*\tens{B}\\
&=\tens{E}^{\tens{A}}*\left(  \frac{1}{2 \pi {\textbf i}}\int_{\Gamma_1} \widehat{\tens{R}}(\lambda,\tens{A})\ {\rm d}\lambda  \right)*\tens{D}*\left(  \frac{1}{2 \pi {\textbf i}}\int_{\Gamma_2} \widehat{\tens{R}}(\mu,\tens{B})\ {\rm d}\mu  \right)*\tens{E}^{\tens{B}}\\
&= \tens{E}^{\tens{A}}*(\tens{E}^{\tens{A}})^{\top}*\tens{D}*(\tens{E}^{\tens{B}})^{\top}*\tens{E}^{\tens{B}}\\
&={\rm P}_{\tens{R}(\tens{A})}*\tens{D}*{\rm P}_{\tens{R}(\tens{B})}\\
&=\tens{A}*\tens{A}^{\dag}*\tens{D}*\tens{B}^{\dag}*\tens{B}\\
&=\tens{D}.
\end{aligned}
$$
\end{proof}
At the end of this subsection, we mention the least squares problem based on the T-product,
$$
\min_{x}\norm{\tens{A}*\tens{X}-\tens{B}}_F
$$
the least squares solution to this problem is
$
\tens{X}_{LS}=\tens{A}^{\dag}*\tens{B}.
$
This is the solution to many problems in image processing \cite{Kilmer1, Liu2, Soltani1}. By using the Cauchy integral formula, we have
$$
\tens{X}_{LS}=\frac{1}{2\pi {\textbf  i}}\int_{\Gamma} \frac{1}{z} (z \tens{E}-\tens{A})^{\dag}*\tens{B}\ {\rm d}z.
$$\par
For standard tensor function and an F-Square tensor $\tens{A}$, we can get the similar results as Higham (\cite{Higham1}, Chapter13) as follows:
$$
f(\tens{A})=\frac{1}{2\pi\textbf{i}}\int_{\Gamma}f(z)(z\tens{I}-\tens{A})^{-1} \ {\rm d}z,
$$
and if $y=f(\tens{A})*b$ then
$$
y=\frac{1}{2\pi\textbf{i}}\int_{\Gamma}f(z)(z\tens{I}-\tens{A})^{-1}*b \ {\rm d}z,
$$
where f is analytic on and inside a closed contour $\Gamma$ that encloses the T-eigenvalues \cite{3Miao1} of tensor $\tens{A}$.
\subsection{Generalized tensor power}
\begin{definition}\label{def1-7} {\rm (Generalized tensor power)}
Suppose $\tens{A}$ is an $m\times n\times p$ real-valued tensor and the T-CSVD of $\tens{A}$ is
$$
\tens{A}=\tens{U}_{(r)} * \tens{S}_{(r)} * \tens{V}^{H}_{(r)}.
$$
The generalized power $\tens{A}^{(k)}$ of $\tens{A}$ is defined as:
$$
\tens{A}^{(k)}=\tens{A}^{(k-1)}*\tens{E}^{\top}*\tens{A}, \ k\geq 1,
$$
where
$$
\tens{A}^{(0)}=\tens{E}=\tens{U}_{(r)} *\tens{V}^{H}_{(r)}.
$$
\end{definition}
From Definition \ref{def1-7}, it is easy to obtain the following results.
\begin{corollary}\label{cor1-9}
Suppose $\tens{A}$ is an $m\times n\times p$ real-valued tensor, then
$$
(\tens{A})^{2k+1}=(\tens{A}*\tens{A}^{\top})^k *\tens{A},
$$
$$
(\tens{A})^{2k}=(\tens{A}*\tens{A}^{\top})^k *\tens{E}.
$$
\end{corollary}
Now, the Taylor expansion of a function can be extended to the generalized tensor function.
\begin{theorem}\label{the1-6}
Suppose $\tens{A}$ is an $m\times n\times p$ real-valued tensor, $f: \mathbb{C}\rightarrow \mathbb{C}$ is a scalar function given by,
$$
f(z)=\sum_{k=0}^{\infty}\frac{f^{(k)}(z_0)}{k!}(z-z_0)^k
$$
for
$$
|z-z_0|<R.
$$
Then the generalized tensor function $f^{\Diamond}: \mathbb{C}^{m\times n\times p} \rightarrow \mathbb{C}^{m\times n\times p}$ is given by
$$
f^{\Diamond}(\tens{A})=\sum_{k=0}^{\infty}\frac{f^{(k)}(z_0)}{k!}(\tens{A}-z_0\tens{E})^k,
$$
for
$$
|z_0-c_j^{(i)}|<R, \ i=1,2,\ldots ,p, \ j=1,2,\ldots, r_i.
$$
\end{theorem}
\begin{proof}
By induction, it is easy to verify the equation
$$
(\tens{A}-z_0\tens{E})^{k}=\tens{U}_{(r)}*(S_{(r)}-z_0 \tens{I})^k*\tens{V}_{(r)}^{H}, \ k=0,1,\ldots,
$$
For $n=0,1,\ldots$ we define
$$
\begin{aligned}
f_n^{\Diamond}(\tens{A})&=\sum_{k=0}^{n}\frac{f^{(k)}(z_0)}{k!}(\tens{A}-z_0\tens{E})^k\\
&=\tens{U}_{(r)}*\left(\sum_{k=0}^{n}\frac{f^{(k)}(z_0)}{k!}(\tens{S}_{(r)}-z_0\tens{I})^k\right)*\tens{V}_{(r)}^{H},
\end{aligned}
$$
so that
$$
\norm{f^{\Diamond}(\tens{A})-f_n^{\Diamond}(\tens{A})}\leq \norm{\tens{U}_{(r)}}\norm{\left(\sum_{k=n+1}^{\infty}\frac{f^{(k)}(z_0)}{k!}(\tens{S}_{(r)}
-z_0\tens{I})^k\right)}\norm{\tens{V}_{(r)}^{H}},
$$
by Definition \ref{def1-6} of tensor norm based on the T-SVD, we obtain
$$
\norm{f^{\Diamond}(\tens{A})-f_n^{\Diamond}(\tens{A})}\rightarrow 0,\  \  (n\rightarrow \infty).
$$
\end{proof}
\begin{example}
Suppose $\tens{A}$ is an $m\times n\times p$ real-valued tensor, $f: \mathbb{C}\rightarrow \mathbb{C}$ is a scalar function given by the Laurant expansion,
$$
f(z)=\sum_{k=0}^{\infty}\frac{f^{(k)}(0)}{k!}(z)^k.
$$
Define the functions $f_1(z)$ and $f_2(z)$ as follows:
$$
f_1(z)=\sum_{k=0}^{\infty}\frac{f^{(2k)}(0)}{(2k)!}(z)^{k},\
f_2(z)=\sum_{k=0}^{\infty}\frac{f^{(2k+1)}(0)}{(2k+1)!}(z)^{k},
$$
So it comes to
$$
f^{\Diamond}(\tens{A})=f_1^{\Diamond}(\tens{A}*\tens{A}^{\top})*\tens{E}+f_{2}^{\Diamond}(\tens{A}*\tens{A}^{\top})*\tens{A}.
$$
We have the following Maclaurin formulae for explicit functions {\rm\cite{Golub1, Higham1}},
$$
\exp(z)=\sum_{k=0}^{\infty}\frac{1}{k!}z^k,\
\ln(1+z)=\sum_{k=0}^{\infty}\frac{1}{2k+1}z^{2k+1}-\sum_{k=0}^{\infty}\frac{1}{2k}z^{2k},
$$
$$
\sin(z)=(-1)^k\sum_{k=0}^{\infty}\frac{1}{(2k+1)!}z^{2k+1},
\
\cos(z)=(-1)^k\sum_{k=0}^{\infty}\frac{1}{(2k)!}z^{2k}.
$$
$$
\sinh(z)=\sum_{k=0}^{\infty}\frac{1}{(2k+1)!}z^{2k+1},
\
\cosh(z)=\sum_{k=0}^{\infty}\frac{1}{(2k)!}z^{2k}.
$$
Then generalized tensor function reduces to
$$
\exp^{\Diamond}(\tens{A})=\sum_{k=0}^{\infty}\frac{1}{2k!}(\tens{A}*\tens{A}^{\top})^{k}*\tens{E}+\sum_{k=0}^{\infty}\frac{1}{(2k+1)!}(\tens{A}*\tens{A}^{\top})^{k}*\tens{A},
$$
$$
\ln^{\Diamond}(\tens{I}+\tens{A})=\sum_{k=0}^{\infty}\frac{1}{2k+1}(\tens{A}*\tens{A}^{\top})^{k}*\tens{A}-\sum_{k=0}^{\infty}\frac{1}{2k}(\tens{A}*\tens{A}^{\top})^{k}*\tens{E},
$$
$$
\sin^{\Diamond}(\tens{A})=(-1)^k\sum_{k=0}^{\infty}\frac{1}{(2k+1)!}(\tens{A}*\tens{A}^{\top})^{k}*\tens{A},
$$
$$
\cos^{\Diamond}(\tens{A})=(-1)^k\sum_{k=0}^{\infty}\frac{1}{(2k)!}(\tens{A}*\tens{A}^{\top})^{k}*\tens{E}.
$$
$$
\sinh^{\Diamond}(\tens{A})=\sum_{k=0}^{\infty}\frac{1}{(2k+1)!}(\tens{A}*\tens{A}^{\top})^{k}*\tens{A},
$$
$$
\cosh^{\Diamond}(\tens{A})=\sum_{k=0}^{\infty}\frac{1}{(2k)!}(\tens{A}*\tens{A}^{\top})^{k}*\tens{E}.
$$
\begin{remark}\label{rem1-5}
Since $\exp(0)=1\neq 0$ and $\cos(0)=1\neq 0$, the above generalized tensor function $\exp^{\Diamond}(\tens{A})$ and $\cos^{\Diamond}(\tens{A})$ hold only when $c_{j}^{(i)}\neq 0$ for all $i=1,2,\ldots, p$, $j=1,2,\ldots, r$.
\end{remark}
\end{example}

\subsection{Block tensor multiplication}
In matrix cases, knowledge of the structural structures of generalized matrix function $f(A)$ can lead to more accurate and efficient algorithms such as Toeplitz structures, triangular structures, circulant structures and so on \cite{Fiedler1, Horn1, Horn2}, which will lead to significant savings when computing it. Similar savings may be expected to generalized tensor functions. In the following subsections, we indicate to propose some structures preserved by generalized tensor functions based on the tensor T-product.\par
First, we need the following theorems which is the similar case of Corollary \ref{cor1-2} for complex tensors.
\begin{lemma}\label{lem1-6}
Suppose $\tens{A}\in \mathbb{C}^{m\times n\times p}$ is a third order tensor. Let $f: \mathbb{C}\rightarrow \mathbb{C}$ be a scalar function and let $f^{\Diamond}: \mathbb{C}^{m\times n\times p}\rightarrow \mathbb{C}^{m\times n\times p}$ be the corresponding generalized function of third order tensors. Then \\
{\rm (1)} $ [f^{\Diamond}(\tens{A})]^{H}=f^{\Diamond}(\tens{A}^{H})$.\\
{\rm (2)} Let $\tens{P} \in \mathbb{C}^{m\times m\times p}$
and $\tens{Q} \in \mathbb{C}^{n\times n\times p}$ be two unitary tensors, then $f^{\Diamond}(\tens{P} *\tens{A}* \tens{Q})=\tens{P}*f^{\Diamond}(\tens{A})* \tens{Q}$.\\
\end{lemma}
We also need the concept of block tensor which is corresponded to the tensor T-product.
\begin{definition}\label{def1-8} {\rm (Block tensor based on T-product)} Suppose $\tens{A}\in \mathbb{C}^{n_1 \times m_1 \times p}$, $\tens{B}\in \mathbb{C}^{n_1 \times m_2 \times p}$, $\tens{C}\in \mathbb{C}^{n_2 \times m_1 \times p}$ and $\tens{D}\in \mathbb{C}^{n_2 \times m_2 \times p}$ are four tensors. The block tensor
$$
\begin{bmatrix}
\tens{A}& \tens{B} \\
\tens{C}& \tens{D}
\end{bmatrix}\in \mathbb{C}^{(n_1+n_2)\times (m_1+m_2) \times p}
$$
is defined by compositing the frontal slices of the four tensors.
\end{definition}
By the definition of tensor T-product, matrix block and block matrix multiplication, we give the following result.
\begin{theorem}\label{the1-7} {\rm (Tensor block multiplication based on T-product)} Suppose $\tens{A}_1\in \mathbb{C}^{n_1 \times m_1 \times p}$, $\tens{B}_1\in \mathbb{C}^{n_1 \times m_2 \times p}$, $\tens{C}_1\in \mathbb{C}^{n_2 \times m_1 \times p}$, $\tens{D}_1\in \mathbb{C}_1^{n_2 \times m_2 \times p}$, $\tens{A}_2\in \mathbb{C}^{m_1 \times r_1 \times p}$, $\tens{B}_2\in \mathbb{C}^{m_1 \times r_2 \times p}$, $\tens{C}_2\in \mathbb{C}_1^{m_2 \times r_1 \times p}$ and $\tens{D}_2\in \mathbb{C}_1^{m_2 \times r_2 \times p}$ are complex tensors, then we have
\begin{equation}
\begin{bmatrix}
\tens{A}_1& \tens{B}_1 \\
\tens{C}_1& \tens{D}_1
\end{bmatrix}*
\begin{bmatrix}
\tens{A}_2& \tens{B}_2 \\
\tens{C}_2& \tens{D}_2
\end{bmatrix}=
\begin{bmatrix}
\tens{A}_1*\tens{A}_2+\tens{B}_1*\tens{C}_2&  \tens{A}_1*\tens{B}_2+\tens{B}_1*\tens{D}_2 \\
\tens{C}_1*\tens{A}_2+\tens{D}_1*\tens{C}_2& \tens{C}_1*\tens{B}_2+\tens{D}_1*\tens{D}_2 \\
\end{bmatrix}.
\end{equation}
\end{theorem}
\begin{proof}
For the simplicity of illustration, we only choose $p=2$ and prove the simple case,
$$
\begin{bmatrix}
\tens{A}& \tens{B} \\
\tens{C}& \tens{D}
\end{bmatrix}*
\begin{bmatrix}
\tens{E} \\
\tens{F}
\end{bmatrix}=
\begin{bmatrix}
\tens{A}*\tens{E}+\tens{B}*\tens{F} \\
\tens{C}*\tens{E}+\tens{D}*\tens{F} \\
\end{bmatrix}.
$$
$$
\begin{aligned}
{\rm Left-hand\ Side}&= {\rm fold}\left(\bcirc
\begin{bmatrix}
\tens{A}& \tens{B} \\
\tens{C}& \tens{D}
\end{bmatrix}
{\rm unfold}
\begin{bmatrix}
\tens{E} \\
\tens{F}
\end{bmatrix}
\right)\\
&={\rm fold}\left(
\left[\begin{array}{cc|cc}
A^{(1)}&B^{(1)} &A^{(2)} &B^{(2)} \\
C^{(1)}&D^{(1)} &C^{(2)} &D^{(2)} \\
\hline
A^{(2)} &B^{(2)} &A^{(1)} &B^{(1)} \\
C^{(2)} &D^{(2)} &C^{(1)} &D^{(1)} \\
\end{array}\right]
\left[\begin{array}{cc}
E^{(1)}\\
F^{(1)}\\
\hline
E^{(2)} \\
F^{(2)}
\end{array}\right]
\right)\\
&={\rm fold}\left(
\left[\begin{array}{cc}
A^{(1)}E^{(1)}+B^{(1)}F^{(1)}+A^{(2)}E^{(2)}+B^{(2)}F^{(2)}\\
C^{(1)}E^{(1)}+D^{(1)}F^{(1)}+C^{(2)}E^{(2)}+D^{(2)}F^{(2)}\\
\hline
A^{(2)}E^{(1)}+B^{(2)}F^{(1)}+A^{(1)}E^{(2)}+B^{(1)}F^{(2)} \\
C^{(2)}E^{(1)}+D^{(2)}F^{(1)}+C^{(1)}E^{(2)}+D^{(1)}F^{(2)}
\end{array}\right]
\right).
\end{aligned}
$$
$$
\begin{aligned}
{\rm Right-hand\ Side}&=
\begin{bmatrix}
\fold\left(\bcirc(\tens{A})\unfold(\tens{E})\right)+\fold\left(\bcirc(\tens{B})\unfold(\tens{F})\right)\\
\fold\left(\bcirc(\tens{C})\unfold(\tens{E})\right)+\fold\left(\bcirc(\tens{D})\unfold(\tens{F})\right)\\
\end{bmatrix}\\
&=
\begin{bmatrix}
\fold\left(\bcirc(\tens{A})\unfold(\tens{E})+\bcirc(\tens{B})\unfold(\tens{F})\right)\\
\fold\left(\bcirc(\tens{C})\unfold(\tens{E})+\bcirc(\tens{D})\unfold(\tens{F})\right)\\
\end{bmatrix}\\
&=
\begin{bmatrix}
\fold\left(\left[\begin{array}{c|c}A^{(1)}&A^{(2)} \\ \hline A^{(2)}& A^{(1)} \end{array}\right] \left[\begin{array}{c} E^{(1)}\\\hline E^{(2)} \end{array}\right]+\left[\begin{array}{c|c} B^{(1)}&B^{(2)} \\\hline B^{(2)}& B^{(1)} \end{array}\right] \left[\begin{array}{c}  F^{(1)}\\ \hline F^{(2)} \end{array}\right]\right)\\
\fold\left(\left[\begin{array}{c|c} C^{(1)}&C^{(2)} \\\hline C^{(2)}& C^{(1)} \end{array}\right] \left[\begin{array}{c}E^{(1)}\\ \hline E^{(2)} \end{array}\right]+\left[\begin{array}{c|c} D^{(1)}&D^{(2)} \\\hline D^{(2)}& D^{(1)} \end{array}\right] \left[\begin{array}{c} F^{(1)}\\\hline F^{(2)} \end{array}\right] \right)\\
\end{bmatrix}\\
&={\rm fold}\left(
\left[\begin{array}{cc}
A^{(1)}E^{(1)}+B^{(1)}F^{(1)}+A^{(2)}E^{(2)}+B^{(2)}F^{(2)}\\
C^{(1)}E^{(1)}+D^{(1)}F^{(1)}+C^{(2)}E^{(2)}+D^{(2)}F^{(2)}\\
\hline
A^{(2)}E^{(1)}+B^{(2)}F^{(1)}+A^{(1)}E^{(2)}+B^{(1)}F^{(2)} \\
C^{(2)}E^{(1)}+D^{(2)}F^{(1)}+C^{(1)}E^{(2)}+D^{(1)}F^{(2)}
\end{array}\right]
\right)\\
&={\rm Left-hand\ Side}.
\end{aligned}
$$
\end{proof}
\begin{remark}\label{rem1-6} Theorem \ref{the1-7} is of great importance, since it shows that our definition of block tensor confirms the definition of the tensor T-product, which will have great impact on obtaining the results of tensor T-decomposition theorems.
\end{remark}
\begin{remark}\label{rem1-7}
It should be noticed that, in the proof, the following equations do not hold:
$$
\bcirc\left(
\begin{bmatrix}
\tens{A}& \tens{B} \\
\tens{C}& \tens{D}
\end{bmatrix}
\right)\neq
\begin{bmatrix}
\bcirc(\tens{A})& \bcirc(\tens{B}) \\
\bcirc(\tens{C})& \bcirc(\tens{D})
\end{bmatrix},\
\unfold\left(
\begin{bmatrix}
\tens{E} \\
\tens{F}
\end{bmatrix}
\right)\neq
\begin{bmatrix}
\unfold(\tens{E}) \\
\unfold(\tens{F})
\end{bmatrix}.
$$
\end{remark}
\begin{remark}\label{rem1-8}For non-F-square tensors, if we set
$$
\tens{B}=
\begin{bmatrix}
0& \tens{A}\\
\tens{A}^H &0\\
\end{bmatrix}
$$
then $\tens{B}$ is a F-square tensor. We have that for any real-valued odd function $f$, the induced generalized tensor function satisfies
$$
f(\tens{B})=
\begin{bmatrix}
0& f^{\Diamond}(\tens{A})\\
f^{\Diamond}(\tens{A})^H &0\\
\end{bmatrix}.
$$
Here $f(\tens{B})$ is the standard tensor function of $\tens{B}$ while $f^{\Diamond}(\tens{A})$ is the generalized tensor function of $\tens{A}$.\par
For a general function $f=f_{odd}+f_{even}$ where $f_{odd}$ is the odd part of $f$ and $f_{even}$ is the even part of $f$ in terms of power series
$$
f(z)=\sum_{k=0}^{\infty}a_kz^k=f_{even}(z)+f_{odd}(z)=\sum_{k=0}^{\infty}a_{2k}z^{2k}+\sum_{k=0}^{\infty}a_{2k+1}z^{2k+1}.
$$
Then we have the same result as Arrigo, Benzi, and Fenu \cite{Arrigo1} as follows:
$$
f(\tens{A})=
\begin{bmatrix}
f_{even}(\sqrt{\tens{A}*\tens{A}^H})&f_{odd}^{\diamond}(\tens{A})\\
f_{odd}^{\diamond}(\tens{A})&f_{even}(\sqrt{\tens{A}^H*\tens{A}})
\end{bmatrix},
$$
which is a `mixture'  of the standard tensor function $f_{even}$ and the generalized tensor function $f_{odd}^{\diamond}$.
\end{remark}
We give an example by using both generalized tensor function and block tensor multiplication, which can be viewed as the tensor case of Benzi, Estrada, and Klymko \cite{Benzi2}.
\begin{example}
Let $\tens{B}\in\mathbb{C}^{2m\times 2n \times p}$ be an F-Hermitian tensor
$$
\tens{B}=
\begin{bmatrix}
0&\tens{A}\\
\tens{A}^{H}&0\\
\end{bmatrix}
$$
where $\tens{A}\in\mathbb{C}^{m\times n \times p}$. Then the standard exponential function ${\rm exp}(\tens{B})$ is given by
$$
\begin{aligned}
{\rm exp}(\tens{B})
&=\begin{bmatrix}
\cosh\left(\sqrt{\tens{A}*\tens{A}^{H}}\right)&\sinh^{\diamond}\left(\sqrt{\tens{A}^{H}*\tens{A}}\right)\\
\sinh^{\diamond}\left(\sqrt{\tens{A}^{H}*\tens{A}}\right)&\cosh\left(\sqrt{\tens{A}^{H}*\tens{A}}\right)
\end{bmatrix}\\
&=\begin{bmatrix}
\cosh\left(\sqrt{\tens{A}*\tens{A}^{H}}\right)&\tens{A}*\left(\sqrt{\tens{A}^{H}*\tens{A}}\right)^{\dag}*\sinh\left(\sqrt{\tens{A}^{H}*\tens{A}}\right)\\
\sinh\left(\sqrt{\tens{A}^{H}*\tens{A}}\right)*\left(\sqrt{\tens{A}^{H}*\tens{A}}\right)^{\dag}*\tens{A}^{H}&\cosh\left(\sqrt{\tens{A}^{H}*\tens{A}}\right)
\end{bmatrix},
\end{aligned}
$$
where $\sinh^{\diamond}$ is the generalized tensor function. $\sinh$, $\cosh$ is the standard tensor function. \par
The second equivalent relation is due to Corollary \ref{cor1-4}:
$$
f^{\diamond}(\tens{A})=f\left(\sqrt{\tens{A}*\tens{A}^H}\right)\left(\tens{A}*\tens{A}^H\right)^{\dag}*\tens{A}=\tens{A}*\left(\tens{A}^H*\tens{A}\right)^{\dag}*f\left(\sqrt{\tens{A}^H*\tens{A}}\right).
$$
\end{example}
\subsection{Bilinear forms}
By using the concepts of tensor block and tensor multiplication, now we can define the bilinear forms of tensors, which will lead to further results of invariance of generalized tensor functions and several classes of tensors whose properties are preserved by generalized tensor functions. For matrix cases, there is already a lot of papers on this problem. \par
We emphasize that since the generalized tensor function $f^{\Diamond}(\tens{A})$ cannot be represented as a polynomial of the tensor $\tens{A}$, so the associative tensor T-product algebra may not be closed under our definition of generalized tensor functions. Not only the tensor classes already given in Table $1$ and Table $2$, several other classes of tensors are also introduced. First, we introduced the definition of tensor bilinear form as follows.
\begin{definition}\label{def1-9}\cite{Kilmer1} {\rm (Bilinear form of tensors)} Let $\mathbb{K}^{1\times 1\times p}=\mathbb{R}^{1\times 1\times p}$ or $\mathbb{C}^{1\times 1\times p}$, a scalar product of $\mathbb{K}^{n\times 1\times p}$ is a bilinear or sesquilinear form $\langle \cdot, \cdot \rangle_{\tens{T}}$ defined by any nonsingular tensor $\tens{T}\in \mathbb{K}^{n\times n\times p}$ for $x,y \in \mathbb{K}^{n\times 1\times p}$
$$
\langle \cdot, \cdot \rangle_{\tens{T}}:\mathbb{K}^{n\times 1 \times p}\times \mathbb{K}^{n\times 1 \times p}\rightarrow \mathbb{K}^{1\times 1 \times p},
$$
which is given as follows:
\begin{equation}
\langle x, y \rangle_{\tens{T}}=
\left\{
\begin{aligned}
&x^{\top}*\tens{T}*y, \  {\rm for\ real\ or\ complex\ bilinear\ forms},\\
&x^{H}*\tens{T}*y, \  {\rm for\ sesquilinear \ forms}.\\
\end{aligned}
\right.
\end{equation}
\end{definition}
We give a simple example to illustrate the Bilinear form of tensors based on the T-product.
\begin{example}Suppose $x,y\in \mathbb{R}^{3\times 1\times 2}$ is two real tensors whose elements of the frontal slices are given as:
$$
x^{(1)}=\begin{bmatrix}
\frac{3}{2}\\
\frac{3}{2}\\
\frac{5}{2}\\
\end{bmatrix},\quad
x^{(2)}=\begin{bmatrix}
\frac{1}{2}\\
-\frac{1}{2}\\
-\frac{1}{2}\\
\end{bmatrix},
y^{(1)}=\begin{bmatrix}
\frac{3}{2}\\
\frac{5}{2}\\
\frac{1}{2}\\
\end{bmatrix},\quad
y^{(2)}=\begin{bmatrix}
\frac{3}{2}\\
\frac{3}{2}\\
\frac{1}{2}\\
\end{bmatrix}.
$$
We choose $\tens{T}$ as the identity tensor $\tens{I}$. Then $\langle x,y\rangle_{\tens{I}}$ is a $1\times1\times 2$ real tensor, and it is easy to get its frontal slices are
$$
\langle x,y\rangle_{\tens{I}}^{(1)}=7,\quad\langle x,y\rangle_{\tens{I}}^{(2)}=5.
$$
\end{example}
The adjoint of $\tens{A}$ with respect to the scalar product $\langle \cdot, \cdot \rangle_{\tens{T}}$, denoted by $\tens{A}^{\star}$, is uniquely defined by the property $\langle \tens{A}*x, y \rangle_{\tens{T}}=\langle x, \tens{A}^{\star}*y \rangle_{\tens{T}}$ for all $x, y \in \mathbb{K}^{n\times 1\times p}$.\par
Associated with $\langle \cdot, \cdot \rangle_{\tens{T}}$  is an automorphism group $\mathbb{G}$, a Lie algebra $\mathbb{L}$, and a Jordan algebra $\mathbb{J}$, which are the subsets of $\mathbb{K}^{n\times n \times p}$ defined by
$$
\begin{cases}
\mathbb{G}:=\{\tens{G}: \langle \tens{G}*x, \tens{G}*y \rangle_{\tens{T}}= \langle x, y \rangle_{\tens{T}}, \forall x,y\in \mathbb{K}^{n\times 1\times p}\}=\{\tens{G}:\tens{G}^{\star}=\tens{G}^{-1}\},\\
\mathbb{L}:=\{\tens{L}: \langle \tens{L}*x, y \rangle_{\tens{T}}= -\langle x, \tens{L}*y \rangle_{\tens{T}}, \forall x,y\in \mathbb{K}^{n\times 1\times p}\}=\{\tens{L}:\tens{L}^{\star}=-\tens{L}\},\\
\mathbb{J}:=\{\tens{S}: \langle \tens{S}*x, y \rangle_{\tens{T}}= \langle x, \tens{S}*y \rangle_{\tens{T}}, \forall x,y\in \mathbb{K}^{n\times 1\times p}\}=\{\tens{S}:\tens{S}^{\star}=\tens{S}\}.
\end{cases}
$$
Here, $\mathbb{G}$ is a multiplicative group, while $\mathbb{L}$ and $\mathbb{J}$ are linear subspaces of $\mathbb{K}^{n\times n \times p}$.\par
Similar to matrix cases \cite{Benzi1}, we present several important structures for tensors. These tensors form Lie and Jordan tensors algebras over real and complex fields and can be named in terms of symmetry, anti-symmetry and so on, with respect to tensor bilinear and sesquilinear forms. We first introduce some special notations of tensors:\\
{\rm (1)} The reverse tensor $\tens{R}_n\in \mathbb{R}^{n\times n \times p}$ is the tensor whose first frontal slice is
$$
\begin{bmatrix}
& &1\\
& \iddots & \\
1& &
\end{bmatrix}
$$
and other frontal slices are all zeros.\\
{\rm (2)} The skew Hamiltonian tensor
$$
\tens{J}=
\begin{bmatrix}
0&\tens{I}_{nnp} \\
-\tens{I}_{nnp}&0
\end{bmatrix}\in \mathbb{R}^{2n\times 2n \times p}.
$$
{\rm (3)} The pseudo-symmetric tensor
$$
\Sigma_{a,b}=
\begin{bmatrix}
\tens{I}_{aap}&0 \\
0&-\tens{I}_{bbp}
\end{bmatrix}\in \mathbb{R}^{n\times n \times p},
$$
with $a+b=n$.\\
{\rm (4)} The adjoint tensor
\begin{equation}
\tens{A}^{\star}=
\left\{
\begin{aligned}
&\tens{T}^{-1}*\tens{A}^{\top}*\tens{T}, \  {\rm for\ bilinear\ forms},\\
&\tens{T}^{-1}*\tens{A}^{H}*\tens{T}, \  {\rm for\ sesquilinear\ forms},\\
\end{aligned}
\right.
\end{equation}
where $\tens{T}$ is one of the tensors defining the above bilinear or sesquilinear forms. \par
We find the generalized tensor function preserves some tensor structures as the following theorem.
\begin{theorem}\label{the1-8}
Let $\mathbb{T}$ be one of the classes in Table $1$. If $\ten{A}\in \mathbb{T}$ and $f^{\Diamond}(\tens{A})$ is well defined, then $f^{\Diamond}(\tens{A})\in \mathbb{T}$.
\end{theorem}
\begin{proof}
It is obvious that $\tens{R}_n$, $\tens{J}$ and $\Sigma_{a,b}$ are unitary tensors based on the T-product. By Lemma \ref{lem1-6}, we have
$$
f^{\Diamond}(\tens{A}^{\star})=
\left\{
\begin{aligned}
&\tens{T}^{-1}*f^{\Diamond}(\tens{A})^{\top}*\tens{T}=f^{\Diamond}(\tens{A})^{\star}, \  {\rm for\ bilinear\ forms},\\
&\tens{T}^{-1}*f^{\Diamond}(\tens{A})^{H}*\tens{T}=f^{\Diamond}(\tens{A})^{\star}, \  {\rm for\ sesquilinear\ forms}.\\
\end{aligned}
\right.
$$
Hence, for Jordan algebra $\mathbb{J}$ we have $f^{\Diamond}(\tens{A})^{\star}=f^{\Diamond}(\tens{A}^{\star})=f^{\Diamond}(\tens{A})$, for Lie algebra $\mathbb{L}$ we have $f^{\Diamond}(\tens{A})^{\star}=f^{\Diamond}(\tens{A}^{\star})=f^{\Diamond}(-\tens{A})=-f^{\Diamond}(\tens{A})$ because of Remark \ref{rem1-1}, which means we can set $f$ to be an odd function.
\end{proof}
\begin{table}[htbp]
	\centering
	\caption{Structured tensors associated with certain bilinear and sesquilinear forms}
	\label{table1}
	\begin{tabular}{|c|c|c|c|}
		\hline
Space & $\tens{T}$ &Jordan Algebra $\mathbb{J}=\{\tens{S}:\tens{S}^{\star}=\tens{S} \}$  & Lie algebra $\mathbb{L}=\{\tens{L}:\tens{L}^{\star}=-\tens{L}\}$\\
		\hline
		\multicolumn{4}{|c|}{Bilinear Forms} \\
		\hline
$\mathbb{R}^{n\times 1 \times p} $  &  $\tens{I}$    & Symmetrics      & Skew-symmetrics \\
		\hline
$\mathbb{C}^{n\times 1 \times p} $  &  $\tens{I}$    & Complex symmetrics      & Complex skew-symmetrics \\
		\hline
$\mathbb{R}^{n\times 1 \times p} $  &  $\Sigma_{a,b}$    & Pseudo-symmetrics      & Pseudo-skew-symmetrics \\
                \hline
$\mathbb{C}^{n\times 1 \times p} $  &  $\Sigma_{a,b}$    & Complex pseudo-symmetrics      & Complex pseudo-skew-symmetrics \\
                \hline
$\mathbb{R}^{n\times 1 \times p} $  &  $\tens{R}_n$    & Persymmetrics      & Perskew-symmetrics \\
                \hline
$\mathbb{R}^{2n\times 1 \times p} $  &  $\tens{J} $    & Skew-Hamiltonians      & Hamiltonians \\
	       \hline
$\mathbb{C}^{2n\times 1 \times p} $  &  $\tens{J} $    & Complex $\tens{J}$-skew-symmetrics      & Complex $\tens{J}$-symmetrics \\
		\hline
		\multicolumn{4}{|c|}{Sesquilinear Forms} \\	
		\hline
$\mathbb{C}^{n\times 1 \times p} $  &  $\tens{I}$    & Hermitians      & Skew-Hermitians \\  		
	         \hline
$\mathbb{C}^{n\times 1 \times p} $  &  $\Sigma_{a,b}$    & Pseudo-Hermitians      & Pseudo-skew-Hermitians \\
	         \hline
$\mathbb{C}^{n\times 1 \times p} $  &  $\tens{R}_{n}$    & Perhermitians      & Skew-perhermitians \\
	         \hline
$\mathbb{C}^{2n\times 1 \times p} $  &  $\tens{J}$    & $\tens{J}$-skew-Hermitians      & $\tens{J}$-Hermitians \\
                 \hline
			
	\end{tabular}
\end{table}
There are other kind of tensor classes that are preserved by arbitrary generalized tensor functions.\par
\begin{definition}\label{def1-10} {\rm (Centrohermitian tensor)}
$\tens{A}\in \mathbb{C}^{m\times n\times p}$ is centrohermitian (Skew-centrohermitian), if $\tens{R}_m *\tens{A}*\tens{R}_n=\overline{\tens{A}}$ (respectively, $\tens{R}_m *\tens{A}*\tens{R}_n=-\overline{\tens{A}}$).
\end{definition}
\begin{theorem}\label{the1-9}
Suppose $\tens{A}\in \mathbb{C}^{n\times m\times p}$ is a tensor. Let $f: \mathbb{C}\rightarrow \mathbb{C}$ be a scalar function and let $f^{\Diamond}: \mathbb{C}^{m\times n\times p}\rightarrow \mathbb{C}^{m\times n\times p}$ be the corresponding generalized function of third order tensors which is assumed to be well defined at $\tens{A}$.\\
{\rm (1)} If $\tens{A}$ is centrohermitian (skew-centrohermitian), then $f^{\Diamond}(\tens{A})$ is also centrohermitian (Skew-centrohermitian).\\
{\rm (2)} If $m=n$ and $\ten{A}^H*\tens{A}=\ten{A}*\tens{A}^H$ (which is called normal), then $f^{\Diamond}(\tens{A})$ is also normal.\\
{\rm (3)} If $m=n$ and $\ten{A}$ is F-circulant then $f^{\Diamond}(\tens{A})$ is also F-circulant.\\
{\rm (4)} If $\ten{A}$ is a F-block-circulant tensor with F-circulant blocks, then $f^{\Diamond}(\tens{A})$ is also a F-block-circulant tensor with F-circulant blocks.
\end{theorem}
\begin{proof}
{\rm (1)} If $\tens{A}$ is centrohermitian, then we have $\tens{R}_m *\tens{A}*\tens{R}_n=\overline{\tens{A}}$. Since $\tens{R}_n$ and $\tens{R}_n$ are unitary tensors, so by Lemma \ref{lem1-6}, we have $\tens{R}_m *f^{\Diamond}(\tens{A})*\tens{R}_n=f^{\Diamond}(\tens{R}_m *\tens{A}*\tens{R}_n)=f^{\Diamond}(\overline{\tens{A}})=\overline{f^{\Diamond}(\tens{A})}$.\par
 If $\tens{A}$ is Skew-centrohermitian, then we have $\tens{R}_m *\tens{A}*\tens{R}_n=-\overline{\tens{A}}$. Since $\tens{R}_n$ and $\tens{R}_n$ are unitary tensors, so by Lemma \ref{lem1-6}, we have $\tens{R}_m *f^{\Diamond}(\tens{A})*\tens{R}_n=f^{\Diamond}(\tens{R}_m *\tens{A}*\tens{R}_n)=f^{\Diamond}(-\overline{\tens{A}})=-\overline{f^{\Diamond}(\tens{A})}$.\\
{\rm (2)} Suppose $\tens{A}\in \mathbb{C}^{m\times n\times p}$ is normal, that is, $\tens{A}*\tens{A}^H=\tens{A}^H*\tens{A}$. So we take the `$\bcirc$' operator on both sides of the equation, notice that $\bcirc(\tens{A}^H)=(\bcirc(\tens{A}))^H$, we get $\bcirc(\tens{A})$ is a normal matrix. Suppose
$$
\bcirc(\tens{A})=(F_{p}\otimes I_{m} )
\begin{bmatrix}
 D_1 &  &  &  \\

  & D_2 &   & \\

  & & \ddots &\\

 & &  & D_{p}
\end{bmatrix}
(F_{p}^{H}\otimes I_{n} ).
$$
We get $D_1, D_2,\cdots, D_p$ are normal matrices, by the definition of generalized tensor function, we get the result.\\
{\rm (3)} and {\rm (4)} hold because of the same reason of {\rm (2)} by taking `$\bcirc$' operator on the tensor $\tens{A}$.
\end{proof}
For some structured tensors, their structures may not be preserved by every generalized tensor function $f^{\Diamond}$, but if we put some restrictions on the original scalar function $f$, we can get $f^{\Diamond}(\tens{A})$ holds the same structure with $\tens{A}$. The following theorems are good examples.
\begin{theorem}\label{the1-10}
Let $\mathbb{G}$ be one of the tensor groups in Table $2$. If $\tens{A}\in \mathbb{G}$, $f: \mathbb{R}\rightarrow \mathbb{R}$ is defined for $x>0$ and satisfies
$$
f(x)f(\frac{1}{x})=1\ {\rm and}\ f(0)=0,
$$
then $f^{\Diamond}(\tens{A})\in \mathbb{G}$.
\end{theorem}
\begin{proof}
Since $\tens{R}_n$, $\tens{J}$ and $\Sigma_{a,b}$ are unitary, by Lemma \ref{lem1-6} we have
$$
f^{\Diamond}(\tens{A}^{\star})=
\left\{
\begin{aligned}
&\tens{T}^{-1}*f^{\Diamond}(\tens{A})^{\top}*\tens{T}=f^{\Diamond}(\tens{A})^{\star}, \  {\rm for\ bilinear\ forms},\\
&\tens{T}^{-1}*f^{\Diamond}(\tens{A})^{H}*\tens{T}=f^{\Diamond}(\tens{A})^{\star}, \  {\rm for\ sesquilinear\ forms}.\\
\end{aligned}
\right.
$$
Hence, for $\tens{A}$ in each of the above tensor automorphism groups $\mathbb{G}=\{\tens{G}:\tens{G}^{\star}=\tens{G}^{-1} \}$, we have
$$
\begin{aligned}
&f^{\Diamond}(\tens{A})^{\star}=f^{\Diamond}(\tens{A}^{\star})=f^{\Diamond}(\tens{A}^{-1})\\
=&\ten{V}_r* f^{\Diamond}(\tens{S}^{-1})*\tens{U}_r^H\\
=&\ten{V}_r *f^{\Diamond}(\tens{S})^{-1}*\tens{U}_r^H\\
=&f^{\Diamond}(\tens{A})^{-1}.
\end{aligned}
$$
\end{proof}
\begin{remark}\label{rem1-9} It could be noticed that for any unitary tensor $\tens{A}$, we have $f^{\Diamond}(\tens{A})=f(1)\tens{A}$, therefore $f^{\Diamond}(\tens{A})=\tens{A}$ for any function $f$ satisfying $f(1)=1$. Hence, some generalized tensor functions may be trivial functions for unitary tensors. A complete characterization of all (meromorphic) functions satisfying the condition that $f(x)f(\frac{1}{x})=1$ can be found in \cite{Higham1}.
\end{remark}
\begin{table}[htbp]
	\centering
	\caption{Structured tensors associated with certain bilinear and sesquilinear forms}
	\label{table1}
	\begin{tabular}{|c|c|c|}
		\hline
Space & $\tens{T}$ &Automorphism Group $\mathbb{G}=\{\tens{G}:\tens{G}^{\star}=\tens{G}^{-1} \}$  \\
		\hline
		\multicolumn{3}{|c|}{Bilinear Forms} \\
		\hline
$\mathbb{R}^{n\times 1 \times p} $  &  $\tens{I}$    & Real orthogonals      \\
		\hline
$\mathbb{C}^{n\times 1 \times p} $  &  $\tens{I}$    & Complex orthogonals     \\
		\hline
$\mathbb{R}^{n\times 1 \times p} $  &  $\Sigma_{a,b}$    & Pseudo-orthogonals    \\
                \hline
$\mathbb{C}^{n\times 1 \times p} $  &  $\Sigma_{a,b}$    & Complex pseudo-orthogonals      \\
                \hline
$\mathbb{R}^{n\times 1 \times p} $  &  $\tens{R}_n$    & Real perplectics  \\
                \hline
$\mathbb{R}^{2n\times 1 \times p} $  &  $\tens{J} $    & Real symplectics \\
	       \hline
$\mathbb{C}^{2n\times 1 \times p} $  &  $\tens{J} $    & Complex symplectics     \\
		\hline
		\multicolumn{3}{|c|}{Sesquilinear Forms} \\	
		\hline
$\mathbb{C}^{n\times 1 \times p} $  &  $\tens{I}$    & Unitaries      \\  		
	         \hline
$\mathbb{C}^{n\times 1 \times p} $  &  $\Sigma_{a,b}$    & Pseudo-unitaries    \\
	         \hline
$\mathbb{C}^{n\times 1 \times p} $  &  $\tens{R}_{n}$    & Complex perplectics   \\
	         \hline
$\mathbb{C}^{2n\times 1 \times p} $  &  $\tens{J}$    & Conjugate symplectics      \\
                 \hline
			
	\end{tabular}
\end{table}
\begin{theorem}\label{the1-11}
Let $\tens{A}\in \mathbb{R}^{m\times n \times p}$ be a nonnegative tensor and $f$ be the odd part of an analytic function which has the Laurant expansion of the form $f(z)=\sum_{k=0}^{\infty} c_k z^k$ with $c_{2k+1}\geq 0$, assumed to be convergent for $|z|<R$ with $R>\norm{\tens{A}}_2$. Then $f^{\Diamond}(\tens{A})$ is well-defined, and $f^{\Diamond}(\tens{A})$ is also a nonnegative tensor.
\end{theorem}
\begin{proof}
Without loss of generality, we assume $f$ is an odd function which has the Laurant expansion
$$
f(z)=\sum_{k=0}^{\infty} c_{2k+1} z^{2k+1}
$$
with coefficients $c_{2k+1}\geq 0$. The condition on the radius of convergence of the Laurant series guarantees that $f^{\Diamond}(\tens{A})$ is well-defined. Since $\tens{A}$ has the T-CSVD
$$
\tens{A}=\tens{U}_{(r)}*\Sigma_{(r)}*\tens{V}_{(r)}^{H},
$$
then $\left(\tens{A}*\tens{A}^{\top}\right)^{2k}*\tens{A}=\tens{U}_{(r)}*\Sigma_{(r)}^{2k+1}*\tens{V}_{(r)}^{H}$.
It turns out to be that
$$
f^{\Diamond}(\tens{A})=\tens{U}_{(r)}*\left(\sum_{k=0}^{\infty}c_{2k+1}\Sigma_{(r)}^{2k+1}\right)*\tens{V}_{(r)}^{H}\geq 0.
$$
\end{proof}
\begin{definition}\label{def1-11}{\rm (Permutation tensor)} A tensor $\tens{P}\in \mathbb{R}^{n\times n\times p}$ is called a permutation tensor, if its first frontal slice is a permutation matrix and other frontal slices are all zeros.
\end{definition}
It is obvious that permutation tensors are all orthogonal tensors, i.e.,
$$
\tens{P}*\tens{P}^{\top}=\tens{P}^{\top}*\tens{P}=\tens{I}.
$$\par
The next theorem shows that the generalized tensor function preserves zero slices in certain positions of tensors.
\begin{theorem}\label{the1-12} Let $\tens{A}\in \mathbb{C}^{n\times n \times p}$ be a complex tensor and $f^{\Diamond}(\tens{A})$ be well-defined.\par
{\rm (1)} If the $i$th lateral (horizontal) slice of $\tens{A}$ consists of all zeros, then the $i$-th lateral (horizontal) slice of $f^{\Diamond}(\tens{A})$ consists of all zeros.\par
 {\rm (2)} If there exist permutation tensors $\tens{P}\in \mathbb{R}^{n\times n\times p}$ and $\tens{Q}\in \mathbb{R}^{n\times n\times p}$ such that $\tens{P}*\tens{A}*\tens{Q}$ is F-block diagonal, then $\tens{P}*f^{\Diamond}(\tens{A})*\tens{Q}$ is also a F-block diagonal tensor.
 \end{theorem}
\begin{proof}
{\rm (1)} Without loss of generality, we may assume the last lateral slice of $\tens{A}$ is $0$, since for any permutation tensor $\tens{P}$, we have
$$
f^{\Diamond}(\tens{A}*\tens{P})=f^{\Diamond}(\tens{A})*\tens{P}.
$$
We write $\tens{A}=\left[\widehat{\tens{A}}\ \ \ 0\right]$ and assume that $\widehat{\tens{A}}$ has T-SVD decomposition
$$
\widehat{\tens{A}}=\widehat{\tens{U}}*\widehat{\Sigma}*\widehat{\tens{V}}^H.
$$
It follows this equation that $\tens{A}$ has the T-SVD decomposition via block tensor multiplication
$$
\tens{A}=
\begin{bmatrix}
\widehat{\tens{A}}&0
\end{bmatrix}=
\widehat{\tens{U}}*
\begin{bmatrix}
\widehat{\Sigma}&0
\end{bmatrix}*
\begin{bmatrix}
\widehat{\tens{V}}&0\\
0&{\textbf 1}
\end{bmatrix}=
\tens{U}*\Sigma*\tens{V}^H,
$$
where the `${\textbf 1}$' in the tensor block is a tube tensor whose first element is 1 and the other element are all zeros.\par
We assume that $\tens{A}$ has tubal rank $r$, we have the T-CSVD
$$
f^{\Diamond}(\tens{A})=\tens{U}_{(r)}*f^{\Diamond}(\Sigma_{(r)})*\tens{V}_{(r)}^H.
$$
We can find the last lateral slice of $\tens{V}_{(r)}$ consists of all zeros, so it comes to the conclusion that the last lateral slice of $f^{\Diamond}(\tens{A})$ is also zero.\par
For the similar reason, if $\tens{A}$ has zero horizontal slices, then we can use the same method for $\tens{A}^H$. By using the conclusion that $f^{\Diamond}(\tens{A})^H=f^{\Diamond}(\tens{A}^H)$, we can get the corresponding result.\par
{\rm (2)} By using the tensor block multiplication via T-block, we can easily get the result.
\end{proof}

\section{Isomorphisms and Invariants}
\subsection{Complex-to-real isomorphism}
In this section, we show that the generalized tensor functions are well-behaved with respect to the canonical isomorphism between the algebra of $n\times n\times p $ complex tensors and the subalgebra of the algebra of the real $2n\times 2n\times p$ tensors consisting of all block tensors of the form:
$$
\begin{bmatrix}
\tens{B}& -\tens{C}\\
\tens{C}& \tens{B}
\end{bmatrix},
$$
where $\tens{B}$ and $\tens{C}$ are tensors in $\mathbb{R}^{n\times n \times p}$.
\begin{theorem}\label{the1-13} Let $\tens{A}=\tens{B}+{\textbf i}\ \tens{C}\in \mathbb{C}^{n\times n \times p}$, where $\tens{B}$ and $\tens{C}$ are real tensors. Let $f: \mathbb{R}\rightarrow \mathbb{R}$ be a scalar function satisfies $f(0)=0$. $f^{\Diamond}: \mathbb{C}^{n\times n\times p}\rightarrow \mathbb{C}^{n\times n\times p}$ is the induced generalized tensor function. Let $\Phi: \mathbb{C}^{n\times n\times p}\rightarrow \mathbb{R}^{2n\times 2n\times p}$ be the mapping:
$$
\Phi(\tens{A})=
\begin{bmatrix}
\tens{B}& -\tens{C}\\
\tens{C}& \tens{B}
\end{bmatrix}.
$$
We denote by $f^{\Diamond}$ the generalized tensor function from $\mathbb{R}^{2n\times 2n\times p}$ to $\mathbb{R}^{2n\times 2n\times p}$ induced by $f$. Then $f^{\Diamond}(\Phi(\tens{A}))$ is well defined and $f^{\Diamond}$ commutes with $\Phi$:
\begin{equation}
f^{\Diamond}(\Phi(\tens{A}))=\Phi(f^{\Diamond}(\tens{A})).
\end{equation}
That is to say, we have the following commutative diagram:
$$
\begin{array}{ccc}
\mathbb{C}^{n\times n\times p}\ (*,+, \tens{I}_n,0) & \xrightarrow{\text{$f^{\Diamond}$}}&\mathbb{C}^{n\times n\times p}\ (*,+, \tens{I}_n,0) \\
\Bigg\downarrow{\text{$\Phi$}} && \Bigg\downarrow{\text{$\Phi$}}\\
\mathbb{R}^{2n\times 2n\times p} \ (*,+, \tens{I}_{2n},0)& \xrightarrow{\text{$f^{\Diamond}$}}&\mathbb{R}^{2n\times 2n\times p}\ (*,+, \tens{I}_{2n},0)
\end{array}
$$
\end{theorem}
\begin{proof}
By using the T-SVD of tensors, we have
$$
\begin{aligned}
\tens{A}&=\tens{B}+{\textbf i}\ \tens{C}\\
&=\tens{U}*\Sigma*\tens{V}^H\\
&=(\tens{U}_1+{\textbf i}\ \tens{U}_2)*\Sigma*(\tens{V}_1+{\textbf i}\ \tens{V}_2)^H\\
&=(\tens{U}_1+{\textbf i}\ \tens{U}_2)*\Sigma*(\tens{V}_1^{\top}-{\textbf i}\ \tens{V}_2^{\top})\\
&=\tens{U}_1*\Sigma*\tens{V}_1^{\top}+\tens{U}_2*\Sigma*\tens{V}_2^{\top}+{\textbf i}\ (\tens{U}_2*\Sigma*\tens{V}_1^{\top}-\tens{U}_1*\Sigma*\tens{V}_2^{\top}),\\
\end{aligned}
$$
and
$$
\begin{aligned}
f^{\Diamond}(\tens{A})=&\tens{U}_1*f^{\Diamond}(\Sigma)*\tens{V}_1^{\top}+\tens{U}_2*f^{\Diamond}(\Sigma)*\tens{V}_2^{\top}\\
+&{\textbf i}\  (\tens{U}_2*f^{\Diamond}(\Sigma)*\tens{V}_1^{\top}-\tens{U}_1*f^{\Diamond}(\Sigma)*\tens{V}_2^{\top}).
\end{aligned}
$$\par
Hence, by the above block tensor multiplication theorem, we obtain
$$
\Phi(f^{\Diamond}(\tens{A}))=
\begin{bmatrix}
\tens{U}_1*f^{\Diamond}(\Sigma)*\tens{V}_1^{\top}+\tens{U}_2*f^{\Diamond}(\Sigma)*\tens{V}_2^{\top} & -\tens{U}_2*f^{\Diamond}(\Sigma)*\tens{V}_1^{\top}+\tens{U}_1*f^{\Diamond}(\Sigma)*\tens{V}_2^{\top}\\
\tens{U}_2*f^{\Diamond}(\Sigma)*\tens{V}_1^{\top}-\tens{U}_1*f^{\Diamond}(\Sigma)*\tens{V}_2^{\top} & \tens{U}_1*f^{\Diamond}(\Sigma)*\tens{V}_1^{\top}+\tens{U}_2*f^{\Diamond}(\Sigma)*\tens{V}_2^{\top}\\
\end{bmatrix}.
$$\par
Applying tensor singular value decomposition to $\Phi(\tens{A})=
\begin{bmatrix}
\tens{B}& -\tens{C}\\
\tens{C}& \tens{B}
\end{bmatrix}$,
we have the decomposition:
$$
\begin{bmatrix}
\tens{B}& -\tens{C}\\
\tens{C}& \tens{B}
\end{bmatrix}=
\begin{bmatrix}
\tens{U}_1& -\tens{U}_2\\
\tens{U}_2& \tens{U}_1
\end{bmatrix}*
\begin{bmatrix}
\Sigma& 0\\
0& \Sigma
\end{bmatrix}*
\begin{bmatrix}
\tens{V}_1& -\tens{V}_2\\
\tens{V}_2& \tens{V}_1
\end{bmatrix}^{\top}.
$$\par
Since $\tens{U}$ is a unitary tensor, we obtain
$$(\ten{U}_1+{\textbf i}\ \tens{U}_2)*(\ten{U}_1^{\top}-{\textbf i}\ \tens{U}_2^{\top})=\tens{I},
$$
and
$$
(\ten{U}_1^{\top}-{\textbf i}\ \tens{U}_2^{\top})*(\ten{U}_1+{\textbf i}\ \tens{U}_2)=\tens{I}.
$$
Therefore, $\ten{U}_1*\ten{U}_1^{\top}+\ten{U}_2*\ten{U}_2^{\top}=\tens{I}$ and $\ten{U}_1*\ten{U}_2^{\top}=\ten{U}_2*\ten{U}_1^{\top}$.\par
Thus $\begin{bmatrix}
\tens{U}_1& -\tens{U}_2\\
\tens{U}_2& \tens{U}_1
\end{bmatrix}$ is an orthogonal tensor. Similarly, $\begin{bmatrix}
\tens{V}_1& -\tens{V}_2\\
\tens{V}_2& \tens{V}_1
\end{bmatrix}$ is also an orthogonal tensor.\par
On the other hand, it comes to
$$
\begin{aligned}
f^{\Diamond}(\Phi(\tens{A}))&=\begin{bmatrix}
\tens{U}_1& -\tens{U}_2\\
\tens{U}_2& \tens{U}_1
\end{bmatrix}*
\begin{bmatrix}
f^{\Diamond}(\Sigma)& 0\\
0& f^{\Diamond}(\Sigma)
\end{bmatrix}*
\begin{bmatrix}
\tens{V}_1^{\top}& \tens{V}_2^{\top}\\
-\tens{V}_2^{\top}& \tens{V}_1^{\top}
\end{bmatrix}\\
&=\begin{bmatrix}
\tens{U}_1*f^{\Diamond}(\Sigma)*\tens{V}_1^{\top}+\tens{U}_2*f^{\Diamond}(\Sigma)*\tens{V}_2^{\top} & -\tens{U}_2*f^{\Diamond}(\Sigma)*\tens{V}_1^{\top}+\tens{U}_1*f^{\Diamond}(\Sigma)*\tens{V}_2^{\top}\\
\tens{U}_2*f^{\Diamond}(\Sigma)*\tens{V}_1^{\top}-\tens{U}_1*f^{\Diamond}(\Sigma)*\tens{V}_2^{\top} & \tens{U}_1*f^{\Diamond}(\Sigma)*\tens{V}_1^{\top}+\tens{U}_2*f^{\Diamond}(\Sigma)*\tens{V}_2^{\top}\\
\end{bmatrix}.
\end{aligned}
$$
Therefore,
$$
f^{\Diamond}(\Phi(\tens{A}))=\Phi(f^{\Diamond}(\tens{A})).
$$
\end{proof}
This theorem gives us a transformation between the function $f^{\Diamond}: \mathbb{C}^{n\times n\times p}\rightarrow \mathbb{C}^{n\times n\times p}$ and $f^{\Diamond}:\mathbb{R}^{2n\times 2n\times p}\rightarrow \mathbb{R}^{2n\times 2n\times p}$ induced by the same scalar function $f: \mathbb{C}\rightarrow \mathbb{C}$. This theorem will be very useful to avoid complex number computations of generalized tensor functions. Since the map $\Phi$  is invertible, we also have the following commutative diagram:
$$
\begin{array}{ccc}
\mathbb{C}^{n\times n\times p}\ (*,+, \tens{I}_n,0) & \xrightarrow{\text{$f^{\Diamond}$}}&\mathbb{C}^{n\times n\times p}\ (*,+, \tens{I}_n,0) \\
\Bigg\downarrow{\text{$\Phi$}} &\circlearrowleft& \Bigg\uparrow{\text{$\Phi^{-1}$}}\\
\mathbb{R}^{2n\times 2n\times p} \ (*,+, \tens{I}_{2n},0)& \xrightarrow{\text{$f^{\Diamond}$}}&\mathbb{R}^{2n\times 2n\times p}\ (*,+, \tens{I}_{2n},0)
\end{array}
$$

\subsection{Tensor to matrix isomorphism}
It should be noticed that if we have a tensor $\tens{A}\in \mathbb{C}^{m\times n \times p}$, when $p=1$, our definition of generalized tensor function degenerate to the generalized matrix function. We denote the generalized matrix function to be $f^{\Diamond}$. \par
When $p\neq 1$, we want to establish some isomorphism structures between matrices and tensors which might be useful to transfer generalized tensor function problems to generalized matrix function problems. We have the following theorem which shows the $\bcirc$ operator on tensors is an isomorphism between the tensor space $\mathbb{C}^{m\times n \times p}$ and the matrix space $\mathbb{C}^{mp\times np }$.
\begin{theorem}\label{the1-14}
Let $\tens{A}\in \mathbb{C}^{m\times n \times p}$ be a complex tensor and $f:\mathbb{C}\rightarrow \mathbb{C}$ be a scalar function. Denote $f^{\Diamond}$ to be both the induced generalized tensor function and the generalized matrix function. `$\bcirc$' is the tensor block circulant operator. Then $f^{\Diamond}$ commutes with $\bcirc$:
\begin{equation}
f^{\Diamond}(\bcirc(\tens{A}))=\bcirc(f^{\Diamond}(\tens{A})).
\end{equation}
That is to say, we have the following commutative diagram
$$
\begin{array}{ccc}
\mathbb{C}^{m\times n\times p} \ (*,+, \tens{I}_n,0)& \xrightarrow{\text{$f^{\Diamond}$}}&\mathbb{C}^{m\times n\times p} \ (*,+, \tens{I}_n,0)\\
\Bigg\downarrow{\text{$\bcirc$}} && \Bigg\downarrow{\text{$\bcirc$}}\\
\mathbb{C}^{mp\times np} \ (\cdot,+, I_{np},0)& \xrightarrow{\text{$f^{\Diamond}$}}&\mathbb{C}^{mp\times np}\ (\cdot,+, I_{np},0)
\end{array}
$$
\end{theorem}
\begin{proof}
$$
\begin{aligned}
f^{\Diamond}(\bcirc(\tens{A}))&=f^{\Diamond}\left(
\begin{bmatrix}
 A^{(1)} &  A^{(p)}  &  A^{(p-1)} & \cdots &  A^{(2)}\\

 A^{(2)} &  A^{(1)}  &  A^{(p)} & \cdots &  A^{(3)}\\

\vdots  & \ddots& \ddots & \ddots & \vdots\\

 A^{(p)} &  A^{(p-1)}  &  \ddots & A^{(2)} &  A^{(1)}\\
\end{bmatrix}
\right)\\
&=f^{\Diamond}\left(
(F_p\otimes I_n)
\begin{bmatrix}
D_1&&&\\
&D_2&&\\
&&\ddots&\\
&&&D_p
\end{bmatrix}
(F_p^H\otimes I_m)
\right)\\
&=(F_p\otimes I_n)f^{\Diamond}\left(
\begin{bmatrix}
D_1&&&\\
&D_2&&\\
&&\ddots&\\
&&&D_p
\end{bmatrix}\right)
(F_p^H\otimes I_m)\\
&=
(F_p\otimes I_n)
\begin{bmatrix}
f^{\Diamond}(D_1)&&&\\
&f^{\Diamond}(D_2)&&\\
&&\ddots&\\
&&&f^{\Diamond}(D_p)
\end{bmatrix}
(F_p^H\otimes I_m).
\end{aligned}
$$
On the other hand,
$$
\begin{aligned}
\bcirc(f^{\Diamond}(\tens{A}))&=\bcirc\left(
\bcirc^{-1}\left(
(F_p\otimes I_n)
\begin{bmatrix}
f^{\Diamond}(D_1)&&&\\
&f^{\Diamond}(D_2)&&\\
&&\ddots&\\
&&&f^{\Diamond}(D_p)
\end{bmatrix}
(F_p^H\otimes I_m)
\right)
\right)\\
&=(F_p\otimes I_n)
\begin{bmatrix}
f^{\Diamond}(D_1)&&&\\
&f^{\Diamond}(D_2)&&\\
&&\ddots&\\
&&&f^{\Diamond}(D_p)
\end{bmatrix}
(F_p^H\otimes I_m).
\end{aligned}
$$
So we have $
f^{\Diamond}(\bcirc(\tens{A}))=\bcirc(f^{\Diamond}(\tens{A})).
$
\end{proof}
\begin{remark}\label{rem1-10}
Theorem \ref{the1-14} also shows that the generalized matrix functions defined by Ben-Israel \cite{Ben1} is the degenerate case of our generalized tensor functions. The characteristics and properties of generalized tensor functions also hold for generalized matrix functions.
\end{remark}
Since the `$\bcirc$' operator is an invertible operator, the following commutative diagram holds:
$$
\begin{array}{ccc}
\mathbb{C}^{n\times m\times p} \ (*,+, \tens{I}_n,0)& \xrightarrow{\text{$f^{\Diamond}$}}&\mathbb{C}^{n\times m\times p} \ (*,+, \tens{I}_n,0)\\
\Bigg\downarrow{\text{$\bcirc$}} &\circlearrowleft& \Bigg\uparrow{\text{$\bcirc^{-1}$}}\\
\mathbb{C}^{np\times mp} \ (\cdot,+, I_{np},0)& \xrightarrow{\text{$f^{\Diamond}$}}&\mathbb{C}^{np\times mp}\ (\cdot,+, I_{np},0)
\end{array}
$$
The above diagram shows that we transpose the generalized tensor functions to generalized matrix functions and some algorithms in matrices cases maybe useful upon the transposed matrix problems.\par
\begin{corollary}\label{cor1-10}
Let $\tens{A}\in \mathbb{C}^{n\times n \times p}$ be a complex tensor and $f:\mathbb{C}^{n\times n \times p}\rightarrow \mathbb{C}^{n\times n \times p}$ be the induced standard tensor function \cite{Lund1}. `$\bcirc$' is the tensor block circulant operator. Then the standard tensor function also induced the matrix to tensor isomorphism, that is $f$ commutes with $\bcirc$:
\begin{equation}
f(\bcirc(\tens{A}))=\bcirc(f(\tens{A})).
\end{equation}
That is to say, we have the following commutative diagram
$$
\begin{array}{ccc}
\mathbb{C}^{n\times m\times p} \ (*,+, \tens{I}_n,0)& \xrightarrow{\text{$f$}}&\mathbb{C}^{n\times m\times p} \ (*,+, \tens{I}_n,0)\\
\Bigg\downarrow{\text{$\bcirc$}} && \Bigg\downarrow{\text{$\bcirc$}}\\
\mathbb{C}^{np\times mp} \ (\cdot,+, I_{np},0)& \xrightarrow{\text{$f$}}&\mathbb{C}^{np\times mp}\ (\cdot,+, I_{np},0)
\end{array}
$$
\end{corollary}
\begin{proof}
By the same kind of method as the proof of Theorem \ref{the1-14}.
\end{proof}
In probability theory, a stochastic matrix is a square matrix with all rows and columns summing to 1. They are very useful to describe the transitions of a Markov chain. The stochastic matrix was first developed by Andrey Markov at the beginning of the 20th century and it is found varieties of usage throughout quite a lot of scientific fields, such as probability theory, statistics, finance and linear algebra, as well as computer science and population genetics and so on \cite{Asmussen1}. In this paper, we generalize the concept of doubly stochastic matrices to third order tensors as follows:
\begin{definition}\label{def1-12} {\rm (Doubly F-stochastic tensor)} A tensor $\tens{A}\in \mathbb{R}^{n\times n\times p}$ is called doubly F-stochastic if and only if
\begin{equation}
\tens{A}*{\textbf e}=\tens{A}^{\top}*{\textbf e}={\textbf e},
\end{equation}
where ${\textbf e}\in \mathbb{R}^{n\times 1 \times p}$ is a tensor whose elements are all $1$.
\end{definition}
This definition is to say, if a tensor $\tens{A}\in \mathbb{R}^{n\times n\times p}$ is F-doubly stochastic, the sum of all the elements of its horizontal slices and lateral slices are all $1$ (See Fig. $2$). \par
We can also define the right stochastic tensor (left stochastic tensor) with each lateral (horizontal) slice summing to 1.\par
As a beautiful application and illustration of how to use the above isomorphism theorem, we have introduced the concept of doubly F-stochastic tensor and we will prove this set is invariant under the generalized tensor function. In order to prove this, we need the following lemma.
\begin{lemma}\label{lem1-7} {\rm \cite{Benzi1}}
If $A\in \mathbb{R}^{n\times n}$ is doubly stochastic, $f$ satisfies the same assumptions as in Theorem \ref{the1-11} and $f(1)=1$, then $f^{\Diamond}(A)$ is also doubly stochastic.
\end{lemma}
\begin{theorem}\label{the1-15}
If $\tens{A}\in \mathbb{R}^{n\times n\times p}$ is F-doubly stochastic, $f$ satisfies the same assumptions as above and $f(1)=1$, then $f^{\Diamond}(\tens{A})$ is also F-doubly stochastic.
\end{theorem}
\begin{proof}
Since we have
$$
\tens{A}*{\textbf e} =\tens{A}^{\top}*{\textbf e}={\textbf e},
$$
that is equivalent to the equation
$$
\bcirc(\tens{A}) \unfold({\textbf e})=\bcirc(\tens{A})^{\top}\unfold({\textbf e}) =\unfold({\textbf e}) .
$$
So we have $\bcirc(\tens{A})$ is a stochastic matrix. By the previous Lemma \ref{lem1-7}, we have $f^{\Diamond}(\bcirc(\tens{A}))$ is also a stochastic matrix. Since we have the matrix tensor isomorphism, it comes to $\bcirc(f^{\Diamond}(\tens{A}))$ is a stochastic matrix. That is to say,
$$
\bcirc(f^{\Diamond}(\tens{A})) \unfold({\textbf e})=\bcirc(f^{\Diamond}(\tens{A}))^{\top}\unfold({\textbf e})=\unfold({\textbf e})
$$
which is equivalent to
$$
f^{\Diamond}(\tens{A})*{\textbf e} =f^{\Diamond}(\tens{A})^{\top}*{\textbf e}={\textbf e}.
$$
That is to say $f^{\Diamond}(\tens{A})$ is also a F-doubly stochastic tensor.
\end{proof}
\subsection{Invariant tensor cones}
In the previous subsections, we proposed structures of tensors which is preserved under generalized tensor functions. In this subsection, we will talk about tensor cones which is another type of invariant tensor structure.\par
Let $\tens{U}\in \mathbb{C}^{m\times m \times p}$ and $\tens{V}\in \mathbb{C}^{n\times n \times p}$ be two fixed unitary tensors. Denote the set $\tens{S}_{\tens{U},\tens{V}}$ be the set of $m\times n \times p$ complex tensors of the form:
$$
\tens{A}=\tens{U}*\tens{S}* \tens{V}^{H},
$$
where
$$\bcirc(\tens{S})=
(F_{p}\otimes I_m )
\begin{bmatrix}
 (\Sigma_1)_{r} &  & \\
&  (\Sigma_2)_{r}   & &\\
&  & \ddots &    \\

&  & & (\Sigma_{p})_{r} \\
\end{bmatrix}(F_{p}^{H}\otimes I_n ),
$$
$$
(\Sigma_i)_{r}={\rm diag}(c_1^{(i)},c_2^{(i)},\cdots,c_{r}^{(i)},0,0,\cdots,0)\in \mathbb{R}^{m\times n},
$$
here $r$ is the tubal rank of the tensor $\tens{A}$. The singular values $c_j^{(i)}$ satisfy $c_1^{(i)}\geq c_2^{(i)}\geq \cdots \geq c_r^{(i)}\geq 0$. Then the set $\tens{S}_{\tens{U},\tens{V}}$ is a closed convex cone, i.e. the tensor cone is a closed set under the Euclidean topology. Its interior is the set of all tensors $\tens{A}\in\tens{S}_{\tens{U},\tens{V}}$ whose tubal-rank ${\rm rank}_t(\tens{A})=r$.\par
If $f: \mathbb{R}\rightarrow \mathbb{R}$ is any nonnegative non-increasing function for $x>0$, then because of the definition of generalized tensor function, $\tens{S}_{\tens{U},\tens{V}}$ is invariant under $f^{\Diamond}$. Further, if $f(x)>0$ for $x>0$, the induced function $f^{\Diamond}$ maps the interior of the tensor cone $\tens{S}_{\tens{U},\tens{V}}$ to itself.

\section*{Acknowledgments}
The authors would like to thank Prof. M. Benzi for his preprint \cite{Aurentz1} and
the useful discussions with Prof. C. Ling, Prof.  Ph. Toint, Prof. Z. Huang along with his team members, Dr. W. Ding, Dr. Z. Luo, Dr. X. Wang, and Mr. C. Mo.


\begin{thebibliography}{99}

\bibitem{Arrigo1}
F. Arrigo, M. Benzi, and C. Fenu. Computation of generalized matrix functions. SIAM J. Matrix Anal. Appl. 37 (2016), 836--860.

\bibitem{Asmussen1}
S. R. Asmussen. Applied probability and queues. Second edition. Applications of Mathematics (New York), 51. Stochastic Modelling and Applied Probability. Springer-Verlag, New York, 2003.

\bibitem{Aurentz1}
L. Aurentz, A. P. Austin, M. Benzi, and V. Kalantzis. Stable computation of generalized matrix functions via polynomial interpolation. SIAM J. Matrix Anal. Appl. 40 (2019), 210--234.

\bibitem{Baburaj1}
M. Baburaj and S. N. George. Tensor based approach for inpainting of video containing sparse text. Multimedia Tools and Applications. 78 (2019), 1805--1829.

\bibitem{Ben1}
A. Ben-Israel and T.N.E. Greville. Generalized Inverses Theory and Applications. Wiley, New York, 1974; 2nd edition, Springer, New York, 2003.

\bibitem{Benzi1}
M. Benzi and R. Huang. Some matrix properties preserved by generalized matrix functions. Spec. Matrices. 7 (2019), 27--37.

\bibitem{Benzi2}
M. Benzi, E. Estrada, and C. Klymko. Ranking hubs and authorities using matrix functions. Linear Algebra and its Appl. 438 (2013), 2447--2474.
\bibitem{Buono1}

N. D. Buono, L. Lopez, and T. Politi. Computation of functions of Hamiltonian and skew-symmetric matrices. Math. Comput. Simulation. 79 (2008), 1284--1297.

\bibitem{Chan1}
R. Chan and X. Jin. An Introduction to Iterative Toeplitz Solvers, SIAM, Philadelphia, 2007.

\bibitem{Chan2}
T. Chan, Y. Yang, and Y. Hsuan. Polar $n$-complex and $n$-bicomplex singular value decomposition and principal component pursuit. IEEE Trans. Signal Process. 64 (2016), 6533--6544.

\bibitem{Davis1}
P. J. Davis. Circulant Matrices. 2nd Edition, Chelsea Publishing, New York, 1994.

\bibitem{Ely1}
G. Ely, S. Aeron, N. Hao, et al. 5D seismic data completion and denoising using a novel class of tensor decompositions. Geophysics. 80 (2015), V83--V95.

\bibitem{Fiedler1}
M. Fiedler. Special Matrices and Their Applications in Numerical Mathematics. Second edition. Dover Publications, Inc., Mineola, NY, 2008.

\bibitem{Garoni1}
C. Garoni, S. Serra-Capizzano. Generalized Locally Toeplitz Sequences: Theory and Applications. Vol. I. Springer, Cham, 2017.

\bibitem{Gleich1}
D. F. Gleich, G. Chen, and J. M. Varah. The power and Arnoldi methods in an algebra of circulants.  Numer. Linear Algebra Appl. 20 (2013), 809--831.

\bibitem{Golub1}
G. H. Golub and C. F. Van Loan. Matrix Computations, 4th edition, Johns Hopkins University Press, Baltimore, MD, 2013.

\bibitem{Hao1}
N. Hao, M. E. Kilmer, K. Braman, and R. C. Hoover. Facial recognition using tensor-tensor decompositions. SIAM J. Imaging Sci. 6 (2013), 437--463.

\bibitem{Hawkins1}
J. B. Hawkins and A. Ben-Israel. On generalized matrix functions. Linear Multilinear Algebra. 1 (1973), 163--171.

\bibitem{Higham1}
N. J. Higham. Functions of Matrices: Theory and Computation. SIAM, Philadelphia, 2008.

\bibitem{Higham2}
N. J. Higham, D. S. Mackey, N. Mackey, and F. Tisseur. Functions preserving matrix groups and iterations for the matrix square root. SIAM J. Matrix Anal. Appl. 26 (2005), 849--877.

\bibitem{Higham3}
N. J. Higham. J-orthogonal Matrices: Properties and Generation. SIAM Review, 45 (2003), 504--519.

\bibitem{Hill1}
R. D. Hill, R. G. Bates, and S. R. Waters. On per-Hermitian matrices. SIAM J. Matrix Anal. Appl. 11 (1990), 173--179.

\bibitem{Horn1}
A. R. Horn, C. R. Johnson. Matrix Analysis. Second edition. Cambridge University Press, Cambridge, 2013.

\bibitem{Horn2}
A. R. Horn, C. R. Johnson. Topics in matrix analysis. Corrected reprint of the 1991 original. Cambridge University Press, Cambridge, 1994.

\bibitem{Hu1}
W. Hu, D. Tao, W. Zhang, Y. Xie, and Y. Yang. The twist tensor nuclear norm for video completion. IEEE Trans. Neural Netw. Learn. Syst. 28 (2017), 2961--2973.

\bibitem{Hu2}
W. Hu, Y. Yang, W. Zhang, and Y. Xie. Moving object detection using tensor-based low-rank and saliently fused-sparse decomposition. IEEE Trans. Image Process. 26 (2017), 724--737.

\bibitem{Jin1}
X. Jin. Developments and Applications of Block Toeplitz Iterative Solvers, Science Press, Beijing and Kluwer Academic Publishers, Dordrecht, 2002.

\bibitem{Khaleel1}
H. S. Khaleel, S. V. M. Sagheer, M. Baburaj et al. Denoising of Rician corrupted 3D magnetic resonance images using tensor-SVD. Biomedical Signal Processing and Control. 44 (2018), 82--95.

\bibitem{Kilmer1}
M. E. Kilmer, K. Braman, N. Hao, and R. C. Hoover. Third-order tensors as operators on matrices: a theoretical and computational framework with applications in imaging. SIAM J. Matrix Anal. Appl. 34 (2013), 148--172.

\bibitem{Kilmer2}
M. E. Kilmer and C. D. Martin. Factorization strategies for third-order tensors. Linear Algebra Appl. 435 (2011), 641--658.

\bibitem{Kong1}
Z. Kong, L. Han, X. Liu, and X. Yang. A new 4-D nonlocal transform-domain filter for 3-D magnetic resonance images denoising.  IEEE Trans. Medical Imaging.  37 (2018), 941--954.

\bibitem{Kong2}
H. Kong,  X. Xie, and Z. Lin. $t$-Schatten-$p$ norm for low-rank Tensor recovery. IEEE Journal of Selected Topics in Signal Processing. 12 (2018), 1405--1419.

\bibitem{Lee1}
A. Lee. Centrohermitian and skew-centrohermitian matrices. Linear Algebra Appl. 29 (1980), 205--210.

\bibitem{Liu1}
Y. Liu,  L. Chen, and  C. Zhu.  Improved robust tensor principal component analysis via low-rank core matrix. IEEE Journal of Selected Topics in Signal Processing. 12 (2018), 1378--1389.

\bibitem{Liu2}
X. Liu, S. Aeron, V. Aggarwal, X. Wang, and M. Wu. Adaptive sampling of RF fingerprints for fine-grained indoor localization. IEEE Trans. Mobile Computing. 15 (2016), 2411--2423.

\bibitem{Liu3}
X. Liu and X. Wang. Fourth-order tensors with multidimensional discrete transforms. ArXiv preprint, arXiv:1705.01576, 2017.

\bibitem{Long1}
Z. Long, Y Liu, L. Chen et al. Low rank tensor completion for multiway visual data. Signal Processing. 155 (2019), 301--316.

\bibitem{Lund1}
K. Lund. The tensor $t$-function: a definition for functions of third-order tensors. ArXiv preprint, arXiv:1806.07261, 2018.


\bibitem{Madathil1}
B. Madathil  and S. N. George. Twist tensor total variation regularized-reweighted nuclear norm based tensor completion for video missing area recovery. Information Sciences. 423 (2018), 376--397.

\bibitem{Ma1}
H. Ma, N. Li, P. S. Stanimirovi\'{c}, and V. N. Katsikis. Perturbation theory for Moore-Penrose inverse of tensor via Einstein product. Comp. Appl. Math. https://doi.org/10.1007/s40314-019-0893-6.

\bibitem{Madathil2}
B. Madathil and S. N. George. Dct based weighted adaptive multi-linear data completion and denoising. Neurocomputing. 318 (2018), 120--136.

\bibitem{Martin1}
C. D. Martin, R. Shafer, and B. Larue. An order-p tensor factorization with applications in imaging. SIAM J. Sci. Comput. 35 (2013), A474--A490.

\bibitem{3Miao1}
Y. Miao, L. Qi, and Y. Wei. T-Jordan canonical form and T-Drazin inverse based on the T-product. arXiv preprint arXiv:1902.07024 (2019).

\bibitem{Noferini1}
V. Noferini. A formula for the Fr${\rm \acute{e}}$chet derivative of a generalized matrix function. SIAM J. Matrix Anal. Appl. 38 (2017), 434--457.

\bibitem{Qin1}
B. Qin, M. Jin, D. Hao, et al. Accurate vessel extraction via tensor completion of background layer in X-ray coronary angiograms. Pattern Recognition. 87 (2019), 38--54.

\bibitem{Semerci1}
O. Semerci, N. Hao, M. E. Kilmer, and E. L. Miller. Tensor-based formulation and nuclear norm regularization for multienergy computed tomography. IEEE Trans. Image Process. 23 (2014), 1678--1693.

\bibitem{Soltani1}
S. Soltani, M. E. Kilmer, and P. C. Hansen. A tensor-based dictionary learning approach to tomo-graphic image reconstruction. BIT Numerical Mathematics. 56 (2016), 1425--1454.

\bibitem{Sun1}
W. Sun, L. Huang, H. C. So, et al. Orthogonal tubal rank-1 tensor pursuit for tensor completion. Signal Processing. 157 (2019), 213--224.



\bibitem{2Sun1}
L. Sun, B. Zheng, C. Bu, and Y. Wei, Moore-Penrose inverse of tensors via Einstein product. Linear Multilinear Algebra. 64 (2016), 686--698.

\bibitem{Tarzanagh1}
D. A. Tarzanagh and G. Michailidis. Fast randomized algorithms for t-product based tensor operations and decompositions with applications to imaging data. SIAM J. Imag. Science. 11 (2018), 2629--2664.

\bibitem{Wang1}
A. Wang, Z. Lai, and Z. Jin. Noisy low-tubal-rank tensor completion. Neurocomputing. 330 (2019), 267--279.

\bibitem{Weaver1}
J. R. Weaver. Centrosymmetric (cross-symmetric) matrices, their basic properties, eigenvalues, and eigenvectors. The American Mathematical Monthly, 92 (1985), 711--717.

\bibitem{Xie1}
Y. Xie, D. Tao, W. Zhang, Y. Liu, L. Zhang, and Y. Qu. On unifying multi-view self-representations for clustering by tensor multi-rank minimization. Int. J. Comput. Vis. 126 (2018), 1157--1179.


\bibitem{Yang1}
L. Yang, Z. Huang, S. Hu, and J. Han. An iterative algorithm for third-order tensor multi-rank minimization. Comput. Optim. Appl. 63 (2016), 169--202.


\bibitem{Yin1}
M. Yin, J. Gao,   S. Xie,  and  Y. Guo. Multiview subspace clustering via tensorial t-product representation. IEEE Trans. Neural Netw. Learn. Systems, to appear, 2019.

\bibitem{CZhang1}
C. Zhang, W. Hu, T. Jin et al. Nonlocal image denoising via adaptive tensor nuclear norm minimization. Neural Comput. Appl. 29 (2018), 3--19.

\bibitem{Zhang1}
Z. Zhang and S. Aeron. Exact tensor completion using t-SVD. IEEE Trans. Signal Process. 65 (2017), 1511--1526.

\bibitem{Zhang2}
Z. Zhang, G. Ely, S. Aeron, N. Hao, and M. Kilmer. Novel methods for multilinear data completion and denoising based on tensor-svd, Proceeding CVPR '14 Proceedings of the 2014 IEEE Conference on Computer Vision and Pattern Recognition, Pages 3842-3849.

\bibitem{Zhang3}
Z. Zhang. A novel algebraic framework for processing multidimensional data: theory and application. Tufts University, Ph.D thesis, 2017.

\bibitem{Zhou1}
P. Zhou,  C. Lu, Z. Lin, and  C. Zhang. Tensor factorization for low-rank tensor completion.  IEEE Trans. Image Process. 27 (2018), 1152--1163.











\end{thebibliography}
\end{document}